\documentclass[12pt,twoside]{article}
\usepackage{amsmath,amsfonts,amssymb,a4,enumitem,epsfig,array,psfrag}
\usepackage{url}
\usepackage{amsthm}
\usepackage{etoolbox} 
\usepackage{color}
\usepackage{calc}
\usepackage{tikz, graphics} 
\usepackage{stmaryrd} 

\usepackage{comment}

\numberwithin{equation}{section}

\theoremstyle{plain}

\newtheorem{prop}{Proposition}[section]

\newtheorem{lemma}[prop]{Lemma}

\theoremstyle{definition}

\AtEndEnvironment{definition}{\null\hfill$\diamond$}%
\newtheorem{remark}[prop]{Remark}
\AtEndEnvironment{remark}{\null\hfill$\diamond$}%
\newtheorem{example}[prop]{Example}
\AtEndEnvironment{example}{\null\hfill$\diamond$}%

\newcommand{\RR}{{\mathbb{R}}}

\newcommand{\Ss}{{\mathbb{S}}}

\newcommand{\DD}{{\mathbb{D}}}

\newcommand{\cE}{{\cal E}}

\newcommand{\fa}{{\mathfrak a}}

\newcommand{\bF}{{\mathbf F}}

\newcommand{\bQ}{{\mathbf Q}}

\newcommand{\cO}{{\mathcal O}}

\newcommand\thin{\operatorname{thin}}

\newcommand\anti{\operatorname{anti}}

\newcommand\nc{\operatorname{nc}}

\newcommand\GL{\operatorname{GL}}
\newcommand\SL{\operatorname{SL}}
\newcommand\SO{\operatorname{SO}}
\newcommand\Spin{\operatorname{Spin}}
\newcommand\so{\operatorname{\mathfrak{so}}}
\newcommand\spin{\operatorname{\mathfrak{spin}}}

\newcommand\Diag{\operatorname{Diag}}

\newcommand\Quat{\operatorname{Quat}}

\newcommand\Lo{\operatorname{Lo}}
\newcommand\lo{{\mathfrak{lo}}}
\newcommand\Up{\operatorname{Up}}
\newcommand\Cl{\operatorname{Cl}}

\newcommand{\BL}{\operatorname{BL}}
\newcommand{\BLS}{\operatorname{BLS}}
\newcommand{\BLC}{\operatorname{BLC}}
\newcommand{\Bru}{\operatorname{Bru}}
\newcommand\inv{\operatorname{inv}}

\newcommand\sign{\operatorname{sign}}

\newcommand{\Block}{\operatorname{Block}}

\newcommand{\nmesmo}{\llbracket n \rrbracket}
\newcommand{\nmaisum}{\llbracket n+1 \rrbracket}

\newcommand{\medbd}{\makebox[4mm]{$\medblackdiamond$}}
\newcommand{\medwd}{\makebox[4mm]{$\meddiamond$}}
\newcommand{\smabc}{\makebox[4mm]{$\smallblackcircle$}}
\newcommand{\smawc}{\makebox[4mm]{$\smallcircle$}}

\newcommand\fl{{\mathfrak{l}}}

\usetikzlibrary{shapes.geometric,shapes.arrows,decorations.pathmorphing}
\usetikzlibrary{matrix,chains,scopes,positioning,arrows,fit}

\usepackage{fdsymbol}

\begin{document}

\title{Remarks on the homotopy type of \\
intersections of two real Bruhat cells}

\author{Em{\'\i}lia Alves
\footnote{emiliaalves@id.uff.br;
Instituto de Matem\'atica e Estat{\'\i}stica, UFF}
\and Nicolau C. Saldanha
\footnote{saldanha@puc-rio.br,
Departamento de Matem\'atica, PUC-Rio}
}
\maketitle

\begin{abstract}

In~\cite{Alves-Saldanha} we introduce a stratification of intersections of a top dimensional real Bruhat cells with another arbitrary cell.
This intersection is naturally identified with a subset of the lower triangular group $\Lo_{n+1}^1$: 
these subsets are labeled by elements of the finite group $\tilde{B}_{n+1}^+$,
the lift to the spin group of the intersection of the Coxeter group $B_{n+1}$ with $\SO_{n+1}$.
The stratification allows us to determine the homotopy type of the intersection.
In this work we give more examples, particularly detailing the computation of the intersection of two open Bruhat cells in general position for $n=4$. 
It turns out that all connected components are contractible. 
\end{abstract}

\section{Introduction}
Let $\GL_{n+1}$ denote the group of all real invertible
$(n+1) \times (n+1)$ matrices,
$\Up_{n+1} \subset \GL_{n+1}$ the subgroup of upper triangular matrices. 
Let $P_\sigma$ be the permutation matrix with entries
$(P_\sigma)_{i,i^{\sigma}} = 1$ and $0$ otherwise.

For a permutation $\sigma$ in the symmetric group $S_{n+1}$ 
of permutations of $\nmaisum=\{1, 2, \ldots, n, n+1 \}$,
define the Bruhat cell of $\sigma$ in $\GL_{n+1}$ as 
\[ \Bru_\sigma^{\GL} = \{ U_0 P_\sigma U_1;\;U_0, U_1 \in \Up_{n+1} \}
\subset \GL_{n+1}.\]

Let $\Lo_{n+1}^{1}$ denote the unipotent group
of lower triangular matrices with diagonal entries equal to $1$. 
We are interested in the sets 
\[\BL_\sigma = \Lo_{n+1}^1\cap \Bru_\sigma^{\GL}.\] 
The set $\BL_{\sigma}$ is homeomorphic to the intersection of two Bruhat cells, 
namely a top dimensional cell with an arbitrary one.

The study of the intersections of pairs of Bruhat/Schubert cells 
appears naturally in many areas beyond topology such as:
representation theory, singularity theory, Kazhdan-Lusztig theory~\cite{SSV4},~\cite{SSV3},~\cite{SV} and locally convex curves~\cite{Goulart-Saldanha0}. 
See~\cite{Alves-Saldanha} for a longer list of references.

There have been many advances in this topic. For instance, the 
number of connected components of the intersection of two big Bruhat cells in generic position is well known.
In other words: let $\eta$ denote the Coxeter element in the symmetric group $S_{n+1}$.
The number of connected components of $\BL_\eta$ is  $2, 6, 20, 52$ for $n = 1, 2, 3, 4$,
respectively; for $n \geq 5$, the number is 
$3 \cdot 2^n$, see~\cite{SSV1} and~\cite{SSV2}.
This problem of counting connected components can be interpreted 
as counting orbits for certain finite group of sympletic transvections 
acting on a finite-dimensional vector space:
the group and the vector space are uniquely determined by the pairs of Bruhat cells~\cite{Se},~\cite{SSVZ}.
It would be interesting to clarify whether it is possible 
to use the techniques developed in~\cite{SSV1},~\cite{SSV2},~\cite{Se} and~\cite{SSVZ}
to get information about
the low-dimensional homology groups for arbitrary intersections of pairs of Bruhat/Schubert cells.

In~\cite{Alves-Saldanha} we introduce a stratification of the sets $\BL_\sigma$ and 
with the techniques developed in that work we were able to compute the homotopy type of several examples.
In particular, we counted the number of connected components of the sets $\BL_\sigma$, 
for all permutation $\sigma \in S_{n+1}$ and $n = 2, 3, 4$: 
it turns out that all these connected components are contractible, but some long computations are required in the hardest cases.  
The aim of this work is to give more examples and technical remarks that we believe 
can illuminate the understanding of this topic.
In Section~\ref{section:codimtwo} we presents the computations for the hardest example for $n=4$.

\bigskip

We denote a permutation $\sigma$ in many notations.
We use the complete notation, 
a list of the values of $1^\sigma, 2^\sigma, \cdots, (n+1)^\sigma$ enclose in square brackets, for example $\sigma=[45132]$ satisfies 
$1^\sigma = 4$, $2^\sigma = 5$, $3^\sigma = 1$, $4^\sigma = 3$ and $5^\sigma=2$.
Another useful notation is to write $\sigma$ as product of cycles: $\sigma = (143)(25) = [45132]$.
The most important notation for us is to write $\sigma$ as a reduced word. 
Let us denote by $a_1, \cdots, a_n$ the standard Coxeter-Weyl generators of the symmetric
group $S_{n+1}$; where $a_i = (i, i + 1)$.
We denote by $\inv(\sigma)$ the number of inversions of $\sigma$.
A reduced word for $\sigma \in S_{n+1}$ is a product $\sigma = a_{i_1} \cdots a_{i_l}$, where
$\ell=\inv(\sigma)$ is the length of $\sigma$ in the generators $a_i$, $i \in \nmesmo=\{1, \ldots, n\}$.
For instance, we have $\sigma = [45132] = a_2a_1a_3a_2a_4a_3a_2 = a_2a_3a_1a_2a_4a_3a_2$; 
notice that usually there is more than one reduced word.   
A reduced word can be represented by a wiring diagram, as illustrated in Figure~\ref{fig:45132_a}.  
We read left to right and each cross is a generator.

\begin{figure}[ht!]
\centering
\resizebox{0.6\textwidth}{!}{

\begin{tikzpicture}[roundnode/.style={circle, draw=black, fill=black, minimum size=0.1mm,inner sep=2pt},
squarednode/.style={rectangle, draw=red!60, fill=red!5, very thick, minimum size=5mm},
bigdiamondnode/.style={draw,diamond, fill=black, minimum size=1mm,very thick,inner sep=4pt},
diamondnodew/.style={draw,diamond, fill=white, minimum size=1mm,very thick,inner sep=4pt},
bigroundnode/.style={circle, draw=black, fill=black, minimum size=0.1mm,very thick, inner sep=4pt},
bigroundnodew/.style={circle, draw=black, fill=white, minimum size=0.1mm,very thick, inner sep=4pt},]

	
	\begin{scope}[shift={(-5.5,0)}]
		\node[roundnode] at (1,5) (n15) {};
		\node[roundnode] at (1,4) (n14) {};
		\node[roundnode] at (1,3) (n13) {};
		\node[roundnode] at (1,2) (n12) {};
		\node[roundnode] at (1,1) (n11) {};
		
        \node at (2,5) (n25) {};		
		\node at (2,4) (n24) {};
		\node at (2,3) (n23) {};
		\node at (2,2) (n22) {};
		\node at (2,1) (n21) {};

		\node at (3,5) (n35) {};
		\node at (3,4) (n34) {};
		\node at (3,3) (n33) {};
		\node at (3,2) (n32) {};
		\node at (3,1) (n31) {};
		
        \node at (4,5) (n45) {};		
		\node at (4,4) (n44) {};
		\node at (4,3) (n43) {};
		\node at (4,2) (n42) {};
		\node at (4,1) (n41) {};

		\node at (5,5) (n55) {};		
		\node at (5,4) (n54) {};
		\node at (5,3) (n53) {};
		\node at (5,2) (n52) {};
		\node at (5,1) (n51) {};
		
		\node at (6,5) (n65) {};		
		\node at (6,4) (n64) {};
		\node at (6,3) (n63) {};
		\node at (6,2) (n62) {};
		\node at (6,1) (n61) {};

		\node at (7,5) (n75) {};		
		\node at (7,4) (n74) {};
		\node at (7,3) (n73) {};
		\node at (7,2) (n72) {};
		\node at (7,1) (n71) {};

		
		\node[roundnode] at (8,5) (n85) {};
		\node[roundnode] at (8,4) (n84) {};
		\node[roundnode] at (8,3) (n83) {};
		\node[roundnode] at (8,2) (n82) {};
		\node[roundnode] at (8,1) (n81) {};
		
		\draw[very thick] (n15) -- (n25.center);
		\draw[very thick] (n14) -- (n23.center);
		\draw[very thick] (n13) -- (n24.center);
		\draw[very thick] (n12) -- (n22.center);
        \draw[very thick] (n11) -- (n51.center);		
        
        \draw[very thick] (n25.center) -- (n34.center);
        \draw[very thick] (n24.center) -- (n35.center);
        \draw[very thick] (n23.center) -- (n33.center);
        \draw[very thick] (n22.center) -- (n32.center);
        \draw[very thick] (n21.center) -- (n51.center);
        
        \draw[very thick] (n35.center) -- (n85.center);
        \draw[very thick] (n34.center) -- (n44.center);
        \draw[very thick] (n33.center) -- (n42.center);
        \draw[very thick] (n32.center) -- (n43.center);
		
        \draw[very thick] (n44.center) -- (n53.center);
        \draw[very thick] (n43.center) -- (n54.center);
        \draw[very thick] (n42.center) -- (n52.center);

        \draw[very thick] (n54.center) -- (n74.center);
        
        \draw[very thick] (n53.center) -- (n63.center);
        
        \draw[very thick] (n51.center) -- (n62.center);	
        
        \draw[very thick] (n52.center) -- (n61.center);

        
        \draw[very thick] (n64.center) -- (n74.center);
        
        \draw[very thick] (n63.center) -- (n72.center);			
        
        \draw[very thick] (n62.center) -- (n73.center);
        
        \draw[very thick] (n61.center) -- (n81.center);

        
        \draw[very thick] (n74.center) -- (n83.center);	
        
        \draw[very thick] (n73.center) -- (n84.center);	        
        				
		\draw[very thick] (n72.center) -- (n82.center);

	\end{scope}
	
	\begin{scope}[shift={(5.5,0)}]
		\node[roundnode] at (1,5) (n15) {};
		\node[roundnode] at (1,4) (n14) {};
		\node[roundnode] at (1,3) (n13) {};
		\node[roundnode] at (1,2) (n12) {};
		\node[roundnode] at (1,1) (n11) {};
		
        \node at (2,5) (n25) {};		
		\node at (2,4) (n24) {};
		\node at (2,3) (n23) {};
		\node at (2,2) (n22) {};
		\node at (2,1) (n21) {};

		\node at (3,5) (n35) {};
		\node at (3,4) (n34) {};
		\node at (3,3) (n33) {};
		\node at (3,2) (n32) {};
		\node at (3,1) (n31) {};
		
        \node at (4,5) (n45) {};		
		\node at (4,4) (n44) {};
		\node at (4,3) (n43) {};
		\node at (4,2) (n42) {};
		\node at (4,1) (n41) {};

		\node at (5,5) (n55) {};		
		\node at (5,4) (n54) {};
		\node at (5,3) (n53) {};
		\node at (5,2) (n52) {};
		\node at (5,1) (n51) {};
		
		\node at (6,5) (n65) {};		
		\node at (6,4) (n64) {};
		\node at (6,3) (n63) {};
		\node at (6,2) (n62) {};
		\node at (6,1) (n61) {};

		\node at (7,5) (n75) {};		
		\node at (7,4) (n74) {};
		\node at (7,3) (n73) {};
		\node at (7,2) (n72) {};
		\node at (7,1) (n71) {};

		
		\node[roundnode] at (8,5) (n85) {};
		\node[roundnode] at (8,4) (n84) {};
		\node[roundnode] at (8,3) (n83) {};
		\node[roundnode] at (8,2) (n82) {};
		\node[roundnode] at (8,1) (n81) {};
		
		\draw[very thick] (n15) -- (n25.center);
		\draw[very thick] (n14) -- (n23.center);
		\draw[very thick] (n13) -- (n24.center);
		\draw[very thick] (n12) -- (n22.center);
        \draw[very thick] (n11) -- (n51.center);		
        
        \draw[very thick] (n25.center) -- (n35.center);
        \draw[very thick] (n24.center) -- (n34.center);
        \draw[very thick] (n23.center) -- (n32.center);
        \draw[very thick] (n22.center) -- (n33.center);
        
        \draw[very thick] (n35.center) -- (n44.center);
        \draw[very thick] (n34.center) -- (n45.center);
        \draw[very thick] (n33.center) -- (n43.center);
        \draw[very thick] (n32.center) -- (n52.center);
		
        \draw[very thick] (n45.center) -- (n85.center);
        \draw[very thick] (n44.center) -- (n53.center);
        \draw[very thick] (n43.center) -- (n54.center);
        \draw[very thick] (n42.center) -- (n52.center);

        \draw[very thick] (n54.center) -- (n74.center);
        
        \draw[very thick] (n53.center) -- (n63.center);
        
        \draw[very thick] (n51.center) -- (n62.center);	
        
        \draw[very thick] (n52.center) -- (n61.center);

        
        \draw[very thick] (n64.center) -- (n74.center);
        
        \draw[very thick] (n63.center) -- (n72.center);			
        
        \draw[very thick] (n62.center) -- (n73.center);
        
        \draw[very thick] (n61.center) -- (n81.center);

        
        \draw[very thick] (n74.center) -- (n83.center);	
        
        \draw[very thick] (n73.center) -- (n84.center);	        
        				
		\draw[very thick] (n72.center) -- (n82.center);

	\end{scope}
\end{tikzpicture}}
\caption{Examples of reduced words for $\sigma = [45132] \in S_5$. On the left $a_2a_1a_3a_2a_4a_3a_2$, and on the right $a_2a_3a_1a_2a_4a_3a_2$.}
\label{fig:45132_a}
\end{figure}
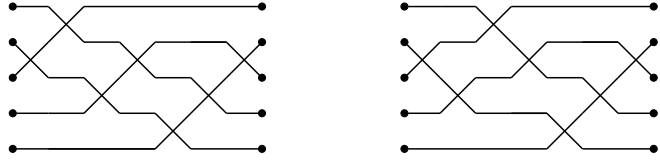	

An ancestry for a reduced word is obtained by marking crossings in the wiring diagram of $\sigma$
with black and white squares and disks. 
Alternatively, an ancestry is a sequence $\varepsilon \in \{\pm 1, \pm 2 \}^\ell$ (where $\ell = \inv(\sigma)$)
satisfying suitable conditions (see Sections 1, 4 and 7 in~\cite{Alves-Saldanha}).

\begin{figure}[ht!]
\centering
\resizebox{0.8\textwidth}{!}{
\begin{tikzpicture}[roundnode/.style={circle, draw=black, fill=black, minimum size=0.1mm,inner sep=2pt},
squarednode/.style={rectangle, draw=red!60, fill=red!5, very thick, minimum size=5mm},
bigdiamondnode/.style={draw,diamond, fill=black, minimum size=1mm,very thick,inner sep=4pt},
diamondnodew/.style={draw,diamond, fill=white, minimum size=1mm,very thick,inner sep=4pt},
bigdiamondnodew/.style={draw,diamond, fill=white, minimum size=1mm,very thick, inner sep=4pt},
bigroundnode/.style={circle, draw=black, fill=black, minimum size=0.1mm,very thick, inner sep=4pt},
bigroundnodew/.style={circle, draw=black, fill=white, minimum size=0.1mm,very thick, inner sep=4pt},]
\begin{scope}[shift={(-9.5,0)}]
		\node[roundnode] at (1,5) (n15) {};
		\node[roundnode] at (1,4) (n14) {};
		\node[roundnode] at (1,3) (n13) {};
		\node[roundnode] at (1,2) (n12) {};
		\node[roundnode] at (1,1) (n11) {};
		
        \node at (2,5) (n25) {};		
		\node at (2,4) (n24) {};
		\node at (2,3) (n23) {};
		\node at (2,2) (n22) {};
		\node at (2,1) (n21) {};

		\node at (3,5) (n35) {};
		\node at (3,4) (n34) {};
		\node at (3,3) (n33) {};
		\node at (3,2) (n32) {};
		\node at (3,1) (n31) {};
		
        \node at (4,5) (n45) {};		
		\node at (4,4) (n44) {};
		\node at (4,3) (n43) {};
		\node at (4,2) (n42) {};
		\node at (4,1) (n41) {};

		\node at (5,5) (n55) {};		
		\node at (5,4) (n54) {};
		\node at (5,3) (n53) {};
		\node at (5,2) (n52) {};
		\node at (5,1) (n51) {};
		
		\node at (6,5) (n65) {};		
		\node at (6,4) (n64) {};
		\node at (6,3) (n63) {};
		\node at (6,2) (n62) {};
		\node at (6,1) (n61) {};

		\node at (7,5) (n75) {};		
		\node at (7,4) (n74) {};
		\node at (7,3) (n73) {};
		\node at (7,2) (n72) {};
		\node at (7,1) (n71) {};

		
		\node[roundnode] at (8,5) (n85) {};
		\node[roundnode] at (8,4) (n84) {};
		\node[roundnode] at (8,3) (n83) {};
		\node[roundnode] at (8,2) (n82) {};
		\node[roundnode] at (8,1) (n81) {};
		
		\draw[very thick] (n15) -- (n25.center);
		\draw[very thick] (n14) -- (n23.center);
		\draw[very thick] (n13) -- (n24.center);
		\draw[very thick] (n12) -- (n22.center);
        \draw[very thick] (n11) -- (n51.center);		
        
        \draw[very thick] (n25.center) -- (n35.center);
        \draw[very thick] (n24.center) -- (n34.center);
        \draw[very thick] (n23.center) -- (n32.center);
        \draw[very thick] (n22.center) -- (n33.center);
        
        \draw[very thick] (n35.center) -- (n44.center);
        \draw[very thick] (n34.center) -- (n45.center);
        \draw[very thick] (n33.center) -- (n43.center);
        \draw[very thick] (n32.center) -- (n52.center);
		
        \draw[very thick] (n45.center) -- (n85.center);
        \draw[very thick] (n44.center) -- (n53.center);
        \draw[very thick] (n43.center) -- (n54.center);
        \draw[very thick] (n42.center) -- (n52.center);

        \draw[very thick] (n54.center) -- (n74.center);
        
        \draw[very thick] (n53.center) -- (n63.center);
        
        \draw[very thick] (n51.center) -- (n62.center);	
        
        \draw[very thick] (n52.center) -- (n61.center);

        
        \draw[very thick] (n64.center) -- (n74.center);
        
        \draw[very thick] (n63.center) -- (n72.center);			
        
        \draw[very thick] (n62.center) -- (n73.center);
        
        \draw[very thick] (n61.center) -- (n81.center);

        
        \draw[very thick] (n74.center) -- (n83.center);	
        
        \draw[very thick] (n73.center) -- (n84.center);	        
        				
		\draw[very thick] (n72.center) -- (n82.center);	
		
			\node[bigroundnodew] at (1.5,3.5) (m1) {};
		\node[bigroundnodew] at (3.5,4.5) (m2) {};
		\node[bigroundnode] at (2.5,2.5) (m3) {};
		\node[bigroundnode] at (4.5,3.5) (m4) {};
		\node[bigroundnode] at (5.5,1.5) (m5) {};
		\node[bigroundnode] at (6.5,2.5) (m6) {};
		\node[bigroundnodew] at (7.5,3.5) (m7) {};

	\end{scope}

		\begin{scope}[shift={(-0.5,0)}]
		\node[roundnode] at (1,5) (n15) {};
		\node[roundnode] at (1,4) (n14) {};
		\node[roundnode] at (1,3) (n13) {};
		\node[roundnode] at (1,2) (n12) {};
		\node[roundnode] at (1,1) (n11) {};
		
        \node at (2,5) (n25) {};		
		\node at (2,4) (n24) {};
		\node at (2,3) (n23) {};
		\node at (2,2) (n22) {};
		\node at (2,1) (n21) {};

		\node at (3,5) (n35) {};
		\node at (3,4) (n34) {};
		\node at (3,3) (n33) {};
		\node at (3,2) (n32) {};
		\node at (3,1) (n31) {};
		
        \node at (4,5) (n45) {};		
		\node at (4,4) (n44) {};
		\node at (4,3) (n43) {};
		\node at (4,2) (n42) {};
		\node at (4,1) (n41) {};

		\node at (5,5) (n55) {};		
		\node at (5,4) (n54) {};
		\node at (5,3) (n53) {};
		\node at (5,2) (n52) {};
		\node at (5,1) (n51) {};
		
		\node at (6,5) (n65) {};		
		\node at (6,4) (n64) {};
		\node at (6,3) (n63) {};
		\node at (6,2) (n62) {};
		\node at (6,1) (n61) {};

		\node at (7,5) (n75) {};		
		\node at (7,4) (n74) {};
		\node at (7,3) (n73) {};
		\node at (7,2) (n72) {};
		\node at (7,1) (n71) {};

		
		\node[roundnode] at (8,5) (n85) {};
		\node[roundnode] at (8,4) (n84) {};
		\node[roundnode] at (8,3) (n83) {};
		\node[roundnode] at (8,2) (n82) {};
		\node[roundnode] at (8,1) (n81) {};
		
		\draw[very thick] (n15) -- (n25.center);
		\draw[very thick] (n14) -- (n23.center);
		\draw[very thick] (n13) -- (n24.center);
		\draw[very thick] (n12) -- (n22.center);
        \draw[very thick] (n11) -- (n51.center);		
        
        \draw[very thick] (n25.center) -- (n35.center);
        \draw[very thick] (n24.center) -- (n34.center);
        \draw[very thick] (n23.center) -- (n32.center);
        \draw[very thick] (n22.center) -- (n33.center);
        
        \draw[very thick] (n35.center) -- (n44.center);
        \draw[very thick] (n34.center) -- (n45.center);
        \draw[very thick] (n33.center) -- (n43.center);
        \draw[very thick] (n32.center) -- (n52.center);
		
        \draw[very thick] (n45.center) -- (n85.center);
        \draw[very thick] (n44.center) -- (n53.center);
        \draw[very thick] (n43.center) -- (n54.center);
        \draw[very thick] (n42.center) -- (n52.center);

        \draw[very thick] (n54.center) -- (n74.center);
        
        \draw[very thick] (n53.center) -- (n63.center);
        
        \draw[very thick] (n51.center) -- (n62.center);	
        
        \draw[very thick] (n52.center) -- (n61.center);

        
        \draw[very thick] (n64.center) -- (n74.center);
        
        \draw[very thick] (n63.center) -- (n72.center);			
        
        \draw[very thick] (n62.center) -- (n73.center);
        
        \draw[very thick] (n61.center) -- (n81.center);

        
        \draw[very thick] (n74.center) -- (n83.center);	
        
        \draw[very thick] (n73.center) -- (n84.center);	        
        				
		\draw[very thick] (n72.center) -- (n82.center);	
		
			\node[bigdiamondnode] at (1.5,3.5) (m1) {};
		\node[bigroundnode] at (3.5,4.5) (m2) {};
		\node[bigroundnode] at (2.5,2.5) (m3) {};
		\node[bigdiamondnodew] at (4.5,3.5) (m4) {};
		\node[bigroundnodew] at (5.5,1.5) (m5) {};
		\node[bigroundnode] at (6.5,2.5) (m6) {};
		\node[bigroundnode] at (7.5,3.5) (m7) {};

	\end{scope}
	
	\begin{scope}[shift={(8.5,0)}]
		\node[roundnode] at (1,5) (n15) {};
		\node[roundnode] at (1,4) (n14) {};
		\node[roundnode] at (1,3) (n13) {};
		\node[roundnode] at (1,2) (n12) {};
		\node[roundnode] at (1,1) (n11) {};
		
        \node at (2,5) (n25) {};		
		\node at (2,4) (n24) {};
		\node at (2,3) (n23) {};
		\node at (2,2) (n22) {};
		\node at (2,1) (n21) {};

		\node at (3,5) (n35) {};
		\node at (3,4) (n34) {};
		\node at (3,3) (n33) {};
		\node at (3,2) (n32) {};
		\node at (3,1) (n31) {};
		
        \node at (4,5) (n45) {};		
		\node at (4,4) (n44) {};
		\node at (4,3) (n43) {};
		\node at (4,2) (n42) {};
		\node at (4,1) (n41) {};

		\node at (5,5) (n55) {};		
		\node at (5,4) (n54) {};
		\node at (5,3) (n53) {};
		\node at (5,2) (n52) {};
		\node at (5,1) (n51) {};
		
		\node at (6,5) (n65) {};		
		\node at (6,4) (n64) {};
		\node at (6,3) (n63) {};
		\node at (6,2) (n62) {};
		\node at (6,1) (n61) {};

		\node at (7,5) (n75) {};		
		\node at (7,4) (n74) {};
		\node at (7,3) (n73) {};
		\node at (7,2) (n72) {};
		\node at (7,1) (n71) {};

		
		\node[roundnode] at (8,5) (n85) {};
		\node[roundnode] at (8,4) (n84) {};
		\node[roundnode] at (8,3) (n83) {};
		\node[roundnode] at (8,2) (n82) {};
		\node[roundnode] at (8,1) (n81) {};
		
		\draw[very thick] (n15) -- (n25.center);
		\draw[very thick] (n14) -- (n23.center);
		\draw[very thick] (n13) -- (n24.center);
		\draw[very thick] (n12) -- (n22.center);
        \draw[very thick] (n11) -- (n51.center);		
        
        \draw[very thick] (n25.center) -- (n35.center);
        \draw[very thick] (n24.center) -- (n34.center);
        \draw[very thick] (n23.center) -- (n32.center);
        \draw[very thick] (n22.center) -- (n33.center);
        
        \draw[very thick] (n35.center) -- (n44.center);
        \draw[very thick] (n34.center) -- (n45.center);
        \draw[very thick] (n33.center) -- (n43.center);
        \draw[very thick] (n32.center) -- (n52.center);
		
        \draw[very thick] (n45.center) -- (n85.center);
        \draw[very thick] (n44.center) -- (n53.center);
        \draw[very thick] (n43.center) -- (n54.center);
        \draw[very thick] (n42.center) -- (n52.center);

        \draw[very thick] (n54.center) -- (n74.center);
        
        \draw[very thick] (n53.center) -- (n63.center);
        
        \draw[very thick] (n51.center) -- (n62.center);	
        
        \draw[very thick] (n52.center) -- (n61.center);

        
        \draw[very thick] (n64.center) -- (n74.center);
        
        \draw[very thick] (n63.center) -- (n72.center);			
        
        \draw[very thick] (n62.center) -- (n73.center);
        
        \draw[very thick] (n61.center) -- (n81.center);

        
        \draw[very thick] (n74.center) -- (n83.center);	
        
        \draw[very thick] (n73.center) -- (n84.center);	        
        				
		\draw[very thick] (n72.center) -- (n82.center);	
		
			\node[bigdiamondnode] at (1.5,3.5) (m1) {};
		\node[bigroundnode] at (3.5,4.5) (m2) {};
		\node[bigdiamondnode] at (2.5,2.5) (m3) {};
		\node[bigroundnode] at (4.5,3.5) (m4) {};
		\node[bigroundnodew] at (5.5,1.5) (m5) {};
		\node[bigdiamondnodew] at (6.5,2.5) (m6) {};
		\node[bigdiamondnodew] at (7.5,3.5) (m7) {};

	\end{scope}
\end{tikzpicture}}
\caption{Examples of ancestries for the reduced word $a_2a_3a_1a_2a_4a_3a_2$ for $\sigma = [45132] \in S_5$:
an ancestry $\varepsilon_0=(+1,-1,+1,-1,-1,-1,+1)$, $\varepsilon_1=(-2,-1,-1,+2,+1,-1,-1)$ 
and $\varepsilon_2=(-2,-1,-1,+2,+1,-1,-1)$ with $\dim(\varepsilon_i)=i$, for $i=0, 1, 2$.}
\label{fig:45132_b}
\end{figure}
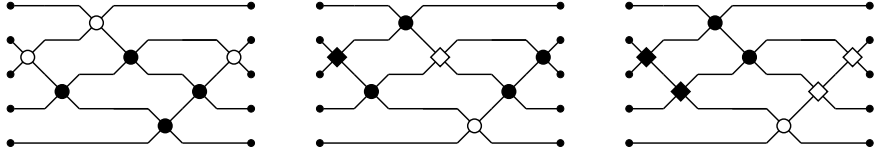	

For each ancestry $\varepsilon$ we define a stratum $\BLS_\varepsilon \subset \BL_\sigma$
that is a smooth contractible submanifold of codimension
\[d = \dim (\varepsilon) = |\{k \,;\, \varepsilon(k) =-2 \}|=|\{k \,;\, \varepsilon(k) =+2  \}|.\]
Further, $\BLS_\varepsilon \subset \BL_\sigma$ is diffeomorphic to $\RR^{\ell-d}$.
Beyond for distinct ancestries the respective strata are disjoint and their union over all ancestries 
is the whole $\BL_\sigma$.
The sets $\BL_\sigma$ are homotopically equivalent to a finite
CW complex $\BLC_\sigma$, with one cell of codimension $d$ for each ancestry of dimension $d$.

\section{Preliminaries}
\label{section:preliminaries}

In this section we briefly review some concepts and results from previous works, 
especially~\cite{Alves-Saldanha} and \cite{Goulart-Saldanha0}.

The spin group $\Spin_{n+1}$ is the double universal cover of $\SO_{n+1}$ and comes with a natural projection $\Pi: \Spin_{n+1} \rightarrow \SO_{n+1}$.
Let $B_{n+1}$ be the Coxeter-Weyl group of signed permutation matrices and set $B_{n+1}^+ = B_{n+1} \cap \SO_{n+1}$. 
Denote by $\Diag_{n+1}^+ \subset B_{n+1}^+$ the normal subgroup of diagonal matrices. 
Defining $\tilde{B}_{n+1}^+ = \Pi^{-1}[B_{n+1}^+]$ and $\Quat_{n+1} = \Pi^{-1}[\Diag_{n+1}^+]$, we have the exact sequences 
\[1 \rightarrow \Quat_{n+1} \rightarrow \tilde{B}_{n+1}^+ \rightarrow S_{n+1} \rightarrow 1, \quad 1 \rightarrow \{ \pm 1 \} \rightarrow \Quat_{n+1} \rightarrow \Diag_{n+1}^+ \rightarrow 1.\]

Let $\fa_i$ be the $(n+1) \times (n+1)$ real skew symmetric matrix with the only non zero entries
being $(\fa_i)_{i+1,i} = 1$ and $(\fa_i)_{i,i+1} = -1$.
Set $\alpha_i(\theta) = \exp(\theta \fa_i)$, the one parameter subgroup $\alpha_i: \RR \to \SO_{n+1}$.
More explicit, $\alpha_i:  \RR \to \SO_{n+1}$ is given by
\[
\alpha_i(\theta) =
\begin{pmatrix} I & & & \\
& \cos(\theta) & -\sin(\theta) & \\
& \sin(\theta) & \cos(\theta) & \\
& & & I \end{pmatrix},\]
where the central block ocupies rows and columns $i$ and $i+1$.
Since $\spin_{n+1}$ and $\so_{n+1}$ are isomorphic we denote by the same symbol 
the one parameter subgroup $\alpha_i: \RR \to \Spin_{n+1}$.
Set $\acute a_j = \alpha_j(\pi/2)$, $\grave a_j = (\acute a_j)^{-1} \in \tilde{B}_{n+1}^+$, 
the elements $\acute a_j$ are generators of the group $ \tilde{B}_{n+1}^+$.
The elements $\hat a_j = (\acute a_j)^2 \in \Quat_{n+1}$, are the generators of the group $\Quat_{n+1}$.

We interpret the spin group $\Spin_{n+1}$ as a subset of the Clifford algebra $\Cl_{n+1}^0$.
This is the subalgebra $\Cl_{n+1}^0 \subset \Cl_{n+1}$ of the even elements of the Clifford algebra $\Cl_{n+1}$.
As in~\cite{Goulart-Saldanha0} and~\cite{Alves-Saldanha} we refer to $\Cl_{n+1}^0$ as the Clifford algebra.
The Clifford algebra $\Cl_{n+1}^0$ is generated by the elements $\hat a_1, \ldots, \hat a_n$.
A basis of the Clifford algebra as a real vector space is
\[ 1, \; \hat a_1, \; \hat a_2, \; \hat a_1 \hat a_2, \; \hat a_3, \; \hat a_1 \hat a_3, \; \hat a_2 \hat a_3, \; \hat a_1 \hat a_2 \hat a_3, \; \hat a_4, \ldots ; \]
this basis is orthonormal for the usual inner product in $\Cl_{n+1}^0$.
For $z \in \Cl_{n+1}^0$ we denote by $\Re(z) = \langle 1,z \rangle \in \RR$ its real part.
In this notation, we have $\alpha_i(\theta) = \cos\left(\frac{\theta}{2}\right) + \sin\left(\frac{\theta}{2}\right) \hat a_i$.

For $\sigma \in S_{n+1}$, let $\Bru_\sigma \subset \Spin_{n+1}$ be the set of $z \in \Spin_{n+1}$ 
such that there exist $U_0, U_1 \in \Up_{n+1}$ with $\Pi(z)=U_0 P_\sigma U_1$, 
where $P_\sigma$ is the permutation matrix.
Each connected component of $\Bru_\sigma$ contains exactly one element of $\tilde{B}_{n+1}^+$;
if $z \in \tilde{B}_{n+1}^+$ belong to a connected component of $\Bru_\sigma$, 
we denote that component by $\Bru_z \subset \Spin_{n+1}$. 
The smooth map $\bQ: \Lo_{n+1}^1 \rightarrow \Spin_{n+1}$ is defined by
$\bQ(I)=1$ and $L = \Pi(\bQ(L))R$, $R \in \Up_{n+1}^+$.
The (possibly empty) subsets $\BL_z \subset \Lo_{n+1}^1$ are defined by $\BL_z = \bQ^{-1}[\Bru_z]$.
The sets $\BL_\sigma$ are partitioned into subsets $\BL_z$:
\begin{equation}
\label{equation:BLz}
\BL_\sigma = \bigsqcup_{z \in \acute\sigma \Quat_{n+1}} \BL_z.
\end{equation}
Here and in~\cite{Alves-Saldanha} we discuss the homotopy type of the sets $\BL_z$:
for $n \leq 4$, the connected components of $\BL_z$ are contractible. 
For this, we constructed here and in~\cite{Alves-Saldanha} a finite stratification of $\BL_z$:
we also obtain a CW-complex $\BLC_z$ with the same homotopy type as $\BL_z$. 
Equation~\eqref{theo:N} below (based on Theorem 4 from~\cite{Alves-Saldanha}) 
gives a formula to count the number of strata with codimension $0$.
More precisely, given $z \in \acute{\sigma} \Quat_{n+1}$ the equation gives a formula for the number $N(z)$
of ancestries $\varepsilon$ with $\dim(\varepsilon)=0$ such that $\BLS_\varepsilon \subseteq \BL_z$.
In particular, it allows us to identify the few special cases when $\BL_z$ is empty.

Following~\cite{Alves-Saldanha}, a permutation $\sigma \in S_{n+1}$ {\em blocks at the entry $k$},
$1 \leq k \leq n $,
if and only if for all $j$,
$j \le k$ implies $j^\sigma \le k$.
Given $\sigma$, let $\Block(\sigma)$
be the set of values of $j$ such that $\sigma$ blocks at $j$. 
Observe that, given a subset  $B \subseteq \nmesmo$, the set $H_B$ of all permutations $\sigma$
such that $\Block(\sigma) \supseteq B$
is the subgroup of $S_{n+1}$ generated by $a_i$, $i \notin B$.
Denote by $\tilde H_B \subseteq \tilde B_{n+1}^{+}$
the subgroup generated by  $\acute a_i$,
$i \in \nmesmo \smallsetminus B$.

Given $\sigma \in S_{n+1}$ and $z \in \acute \sigma \Quat_{n+1} \subset \tilde B_{n+1}^{+}$,
set $\ell = \inv(\sigma)$, $B = \Block(\sigma) \subseteq \nmesmo$, and 
$b = |B|$. In the above notation, 
if $z \notin \tilde H_B$ then $N(z) = N(-z) = 0$;
otherwise
\begin{equation}
N(z) = 2^{\ell-n+b-1} + 2^{\frac{\ell}{2}-1}\,\Re(z). 
\label{theo:N}
\end{equation}
This equation is a restatament of Theorem 4 of~\cite{Alves-Saldanha}.

We denote by $\Diag_{n+1}$ the group of diagonal matrices with diagonal entries in $\{\pm 1\}$.
The quotient $\cE_n = \Diag_{n+1}/\{\pm I\}$ 
is naturally isomorphic to $\{ \pm 1\}^{\nmesmo}$: 
take $D \in \Diag_{n+1}$ to $E \in \{ \pm 1\}^{\nmesmo}$ with 
$E_i=D_{i,i}D_{i+1,i+1}$. 
The group $\Diag_{n+1}$ acts by conjugations in $\SO_{n+1}$: since $-I$ acts trivially, 
this can be considered an action of $\cE_n$.
This action can be lifted to $\Spin_{n+1}$ and then extend to $\Cl_{n+1}^0$.
Specifically, each $E \in \cE_n$ defines automorphisms of $\Spin_{n+1}$ and $\Cl_{n+1}^0$ by
\[ (\hat a_i)^E = E_i \hat a_i, \quad (\alpha_i(\theta))^E = \alpha_i(E_i \theta).  \]
For $z \in \Spin_{n+1}$ we denote by $\cE_z$ the isotropy group of $z$, i.e.,
\[ \cE_z = \{ E \in \cE_n \;|\; z^{E} = z \}. \]

Recall that a set $X \subseteq \nmaisum$ is $\sigma$-invariant
if and only if $X^\sigma = X$,
where $X^\sigma = \{x^\sigma, x \in X\}$.
This happens
if and only if $X$ is a disjoint union of cycles of $\sigma$.
Given $\sigma \in S_{n+1}$,
there exist $2^c$ $\sigma$-invariant sets $X \subseteq \nmaisum$,
where $c = \nc(\sigma)$ is the number of cycles of
the permutation $\sigma$.

\begin{remark}
\label{remark_orbits}
Consider $\sigma \in S_{n+1}$, $z_0 \in \acute{\sigma} \Quat_{n+1}$
and $Q_0 = \Pi(z_0) \in B_{n+1}^+$ (for $\Pi: \Spin_{n+1} \to \SO_{n+1}$).
The orbit $\cO_{Q_0}$ of ${Q_0}$
has cardinality $2^{n-c+1}$.
Concerning the action of $\cE_n$ on $ \acute{\sigma} \Quat_{n+1}$,
there are two possibilities for the size of the orbit $\cO_{z_0}$.
If there exists $E \in \cE_n$
with $z_0^{E} = -z_0$, we set $c_{\anti}(z_0) = 1$;
otherwise $c_{\anti}(z_0) = 0$.
If $c_{\anti}(z_0) = 1$ the orbit $\cO_{z_0}$ is $\Pi^{-1}[\cO_{Q_0}]$,
with cardinality $2^{n-c+2}$.
If $c_{\anti}(z_0) = 0$
the orbits $\cO_{z_0}$ and $\cO_{-z_0}$ are disjoint,
with cardinality $2^{n-c+1}$ and with union $\Pi^{-1}[\cO_{Q_0}]$;
we then say the orbit splits. 
\end{remark}

\section{Strata of codimension zero, one or two}
\label{section:codimone}

In this section we describe the strata of codimension $0$, $1$ or $2$.
This are equivalent to cells of dimension $0$, $1$ or $2$, for which we describe glueing instructions.
This allows us to compute
the connected components of $\BL_\sigma$ or $\BL_z$. \\

We denote by $\lo_{n+1}^1$ the Lie algebra of $\Lo_{n+1}^1$,
i.e., the set of strictly lower triangular matrices.
Let $\fl_j \in \lo_{n+1}^1$, be the matrix whose only nonzero entry is 
$(\fl_j)_{j+1,j} = 1$, for $j \in \nmesmo$.
Denote by $\lambda_j(t)$ its one parameter subgroup, i.e., $\lambda_j(t) = \exp(t \fl_j) \in \Lo_{n+1}^1$.

Given a reduced word $\sigma = a_{i_1} \cdots a_{i_\ell} \in S_{n+1}$ where 
$\ell = \inv(\sigma)$,  
consider the product
\begin{equation}
\label{equation:Lproduct}
 L = \lambda_{i_1}(t_1) \cdots \lambda_{i_\ell}(t_\ell).
\end{equation}
It is well known that if $L \in \BL_{\sigma}$ 
can be written as in  \eqref{equation:Lproduct} then the vector
$(t_1, \ldots, t_\ell)$ is unique~\cite{Berenstein-Fomin-Zelevinsky}.
Also, for almost all $L \in \BL_\sigma \subset \Lo_{n+1}^1$,
there exists a vector $(t_1, \ldots, t_\ell) \in (\RR \smallsetminus \{0\})^\ell$
for which  \eqref{equation:Lproduct} holds. 
If $\varepsilon$ is an ancestry of dimension $0$, then, by definition, 
\[\BL_{\varepsilon} = \{ \lambda_{i_1}(t_1) \cdots \lambda_{i_\ell}(t_\ell);\;
\sign(t_j) = \varepsilon(j) \}
\subseteq \BL_z,\]
where $z=(\acute a_{i_1})^{\sign(\varepsilon(1))} \cdots (\acute a_{i_\ell})^{\sign(\varepsilon(\ell))}$.

\begin{example}
\label{example:45132_definition}

We can take $n=4$ and $\sigma=[45132]$. Let us fix the reduced word $\sigma=a_2a_3a_1a_2a_4a_3a_2$.
If $\varepsilon$ is an ancestry of dimension $0$, then the matrices $L \in \BL_\varepsilon$ can be written as the following product
\begin{equation}
L = \lambda_2(t_1) \lambda_3(t_2) \lambda_1(t_3) \lambda_2(t_4) \lambda_4(t_5) \lambda_3(t_6) \lambda_2(t_7); 
\end{equation}
where $\sign(t_j) = \varepsilon(j)$. 
For instance
\[ L_0 = \begin{pmatrix}
1 & 0 & 0 & 0 & 0 \\ -3 & 1 & 0 & 0 & 0 \\
-3 & -3/2 & 1 & 0 & 0 \\ 0 & -7 & 3 & 1 & 0 \\ 0 & 4 & -2 & -2 & 1
\end{pmatrix} \in \BL_{\varepsilon_0} \subset \BL_{z_0}, \]
where $\varepsilon_0=(+1, +1, -1, -1, -1, +1, -1)$ and then $z_0=\acute a_2 \acute a_3 \grave a_1 \grave a_2 \grave a_4 \acute a_3 \grave a_2$.
\end{example}

But not all matrices in $\BL_\sigma$ can be written as a product of lower matrices given above.
In what follows we will describe how to determine the ancestry a given matrix in $\BL_\sigma$ corresponds to.
As usual, assume $\sigma \in S_{n+1}$ and a reduced word
$\sigma = a_{i_1}\cdots a_{i_\ell}$ to be fixed.
As we saw in~\cite{Alves-Saldanha}, a matrix $L \in \BL_\sigma$ can be identified with $\tilde{z_\ell}=\bQ(L) \in \grave \eta \Bru_{\acute \eta}$ and
with a sequence $(z_k)_{0 \le k \le \ell}$ of elements of $\Spin_{n+1}$
with $z_0 = 1$, $z_k = z_{k-1} \alpha_{i_k}(\theta_k)$,
$\theta_k \in (0,\pi)$, 
$z_\ell = \tilde{z}_\ell q_\ell \in \Bru_{\acute\sigma} \cap (\grave\eta\Bru_{\eta})$, $q_\ell \in \Quat_{n+1}$.
The values of $\theta_k$ are smooth functions of $L \in \BL_\sigma$.
Define $\varrho_k \in \tilde{B}_{n+1}^+$ such that $z_k \in \Bru_{\varrho_k}$.
For all $k$ we have $\varrho_k=\varrho_{k-1} (\acute a_{i_k})^{\xi(k)}$, $\xi(k) \in \{0, 1, 2\}$.
The function $\xi: \{1, \ldots, \ell \} \to \{0, 1, 2\}$ is a corresponding ancestry.\\

We now recall the $\varepsilon$ notation for ancestries.
Let $\rho_k = \Pi(\varrho_k) \in S_{n+1}$, for $\Pi: B_{n+1}^+ \to S_{n+1}$.
For each $k$, define $\tilde{z}_k \in \grave \eta \Bru_{\acute \rho_k}$ by $z_k=\tilde{z}_k q_k, q_k \in \Quat_{n+1}$. 
Define $\tilde \theta_k \in (-\pi, \pi)$ be defined by $\tilde{z}_k=\tilde{z}_{k-1} \alpha_{i_k}(\tilde{\theta}_k)$.
The function $\varepsilon: \{1, \ldots, \ell \} \to \{\pm 1, \pm 2 \}$ is defined by $\sign(\varepsilon(k))=\sign(\tilde{\theta}_k)$ and
$| \varepsilon(k) | = 2$ if and only if $\rho_k \ne \rho_{k-1}$.
The functions  $\xi$ and $\varepsilon$ are alternative descriptions for an ancestry.
Equations (7.1) and (7.2) in~\cite{Alves-Saldanha} teach us how to translate from one notation to the other. \\

Assume $L \in \BL_{\xi} \subset \BL_{\sigma}$,
where $\xi$ is an ancestry of positive dimension $d$.
We can construct a transversal section to $\BL_{\xi}$
by keeping fixed the values of $\theta_k$ if
either $\xi(k) \ne 1$ or $\varrho_k \ge \varrho_{k-1}$.
There are $d$ values of $k$ for which
$\xi(k) = 1$ and $\varrho_k < \varrho_{k-1}$:
for these values of $k$ we allow the coordinates $\theta_k$
to vary freely (and independently)
in a small neighborhood of their original values.
We must then determine the ancestries of the perturbed strata.
An understanding of a transversal section yields
a description of the boundary map.

\begin{lemma}
\label{lemma:codimone}
If $k_1, k_2$ satisfy
\begin{equation}
\label{equation:k1k2}
i_{k_1} = i_{k_2}, \qquad
\forall k, (k_1 < k < k_2) \to (i_k \ne i_{k_1}),
\end{equation}
then
any function $\xi: \llbracket \ell \rrbracket \to \{0,1,2\}$
with $\xi^{-1}[\{1\}] = \{k_1,k_2\}$ is an ancestry of dimension $1$.
Conversely,
if $\xi$ is an ancestry of dimension $d = 1$
then $\xi^{-1}[\{1\}] = \{k_1,k_2\}$ where
$k_1 < k_2$ satisfy the condition in Equation \eqref{equation:k1k2}.
Then there are precisely two ancestries $\tilde\xi$
of dimension $0$ with $\xi \preceq \tilde\xi$.
We can call them $\xi_0, \xi_2$ with $\xi_i(k_1) = i$.
For $k \notin \{k_1,k_2\}$ we have
$\xi_0(k) = \xi_2(k) = \xi(k)$.
The set $\BL_\xi \subset \BL_\sigma$ 
is a submanifold of codimension one
with $\BL_{\xi_0}$ on one side and $\BL_{\xi_2}$ on the other side.
In the CW complex, $\xi$ is represented by an edge from
$\xi_0$ to $\xi_2$.
\end{lemma}

\begin{proof}
The first two claims follow directly from the definition of ancestries.
Let $\xi$ be an ancestry of dimension $d = 1$,
with $k_1 < k_2$ as above.
We then have
\[ \rho_k = \begin{cases}
\eta a_{i_{k_1}},& k_1 \le k < k_2, \\
\eta,& \textrm{otherwise.}
\end{cases} \]
If $\tilde\xi \succ \xi$ we must have
$\tilde\rho_k = \eta$ for $k < k_1$ or $k \ge k_2$:
we thus also have $\tilde\varrho_k = \varrho_k$
and therefore $\tilde\xi(k) = \xi(k)$
for $k < k_1$ or $k > k_2$.
For $k_1 \le k < k_2$ we must have
$\tilde\rho_k \in \{\eta, \eta a_{i_{k_1}}\}$.
If $k_1 \le k-1 < k < k_2$ we have
either $\rho_k = \rho_{k-1}$ or $\rho_k = \rho_{k-1} a_{i_k}$:
the second case contradicts the previous facts.
We thus have $\tilde\rho_k = \eta$ for all $k$.
In particular, $\dim(\tilde\xi) = 0$.

If $w_0, w_1 \in \tilde B_{n+1}^{+}$,
$w_0 < w_1$ and $\Pi(w_0) = \eta a_{i_{k_1}}$
then either
$w_1 = w_0 \acute a_{i_{k_1}}$ or $w_1 = w_0 \grave a_{i_{k_1}}$.
Thus, if $\tilde\xi \succ \xi$ there exists
$\tilde\varepsilon: \{k_1, \ldots, k_2 - 1\} \to \{\pm 1\}$
such that, for all $k$, $k_1 \le k < k_2$ implies
$\tilde\varrho_k =
\varrho_k (\acute a_{i_{k_1}})^{\tilde\varepsilon(k)}$.
For $k_1 < k < k_2$, we have
\begin{align*}
\tilde\varrho_k
&=
\varrho_k (\acute a_{i_{k_1}})^{\tilde\varepsilon(k)}
=
\varrho_{k-1} (\acute a_{i_k})^{\xi(k)}
(\acute a_{i_{k_1}})^{\tilde\varepsilon(k)} \\
&=
\tilde\varrho_{k-1} (\acute a_{i_k})^{\tilde\xi(k)}
=
\varrho_{k-1} (\acute a_{i_{k_1}})^{\tilde\varepsilon(k-1)}
(\acute a_{i_k})^{\tilde\xi(k)}
\end{align*}
and therefore
$(\acute a_{i_k})^{\xi(k)}
(\acute a_{i_{k_1}})^{\tilde\varepsilon(k)} 
=
(\acute a_{i_{k_1}})^{\tilde\varepsilon(k-1)}
(\acute a_{i_k})^{\tilde\xi(k)}$.
If $|i_k - i_{k_1}| = 1$ and $\xi(k) = 2$ 
this implies $\tilde\xi(k) = \xi(k)$ and
$\tilde\varepsilon(k) = - \tilde\varepsilon(k-1)$.
Otherwise,
this implies $\tilde\xi(k) = \xi(k)$ and
$\tilde\varepsilon(k) = \tilde\varepsilon(k-1)$.
In either case, this implies $\tilde\xi(k) = \xi(k)$
for all $k \notin \{k_1,k_2\}$, as desired.
Furthermore, a choice of $\tilde\xi(k_1)$
uniquely determines $\tilde\xi(k)$ for $k_1 < k < k_2$.
Similarly, we have
\[ \tilde\varrho_{k_2}
=
\varrho_{k_2}
=
\varrho_{{k_2}-1} \acute a_{i_{k_1}}
=
\tilde\varrho_{{k_2}-1} (\acute a_{i_{k_2}})^{\tilde\xi({k_2})}
=
\varrho_{{k_2}-1} (\acute a_{i_{k_1}})^{\tilde\varepsilon({k_2}-1)}
(\acute a_{i_{k_1}})^{\tilde\xi({k_2})}
\]
and therefore
$\tilde\xi(k_2) = 1 - \tilde\varepsilon(k_2-1)$,
completing the proof that there are exactly two ancestries
$\tilde\xi$ with $\tilde\xi \succ \xi$.
The other claims follow by construction.
\end{proof}

A $\varepsilon$-ancestry of dimension $0$ can be represented
over a diagram for $\sigma$ by indicating a sign at each intersection.
The edges are then constructed as follows.
A bounded connected component of the complement
of the diagram has vertices 
$k_1$ and $k_2$ on row $i_{k_1}$
plus all vertices $k$ with $k_1 < k < k_2$
and $|i_k - i_{k_1}| = 1$.
If $k_1$ and $k_2$ have opposite signs 
we can click on that connected component,
with the effect of changing all signs on its boundary.

\begin{example}
\label{example:45132-a1}
Figure~\ref{fig:45132x} shows an example of the construction above.
We take $n = 4$ and $\sigma = [45132] = a_2a_3a_1a_2a_4a_3a_2$,
Figure~\ref{fig:45132} shows this reduced word as a diagram.

\begin{figure}[ht!]
\centering
\resizebox{0.2\textwidth}{!}{

\begin{tikzpicture}[roundnode/.style={circle, draw=black, fill=black, minimum size=0.1mm,inner sep=2pt},
squarednode/.style={rectangle, draw=red!60, fill=red!5, very thick, minimum size=5mm},
bigdiamondnode/.style={draw,diamond, fill=black, minimum size=1mm,very thick,inner sep=4pt},
diamondnodew/.style={draw,diamond, fill=white, minimum size=1mm,very thick,inner sep=4pt},
bigroundnode/.style={circle, draw=black, fill=black, minimum size=0.1mm,very thick, inner sep=4pt},
bigroundnodew/.style={circle, draw=black, fill=white, minimum size=0.1mm,very thick, inner sep=4pt},]

	
	\begin{scope}[shift={(-5.5,0)}]
		\node[roundnode] at (1,5) (n15) {};
		\node[roundnode] at (1,4) (n14) {};
		\node[roundnode] at (1,3) (n13) {};
		\node[roundnode] at (1,2) (n12) {};
		\node[roundnode] at (1,1) (n11) {};
		
        \node at (2,5) (n25) {};		
		\node at (2,4) (n24) {};
		\node at (2,3) (n23) {};
		\node at (2,2) (n22) {};
		\node at (2,1) (n21) {};

		\node at (3,5) (n35) {};
		\node at (3,4) (n34) {};
		\node at (3,3) (n33) {};
		\node at (3,2) (n32) {};
		\node at (3,1) (n31) {};
		
        \node at (4,5) (n45) {};		
		\node at (4,4) (n44) {};
		\node at (4,3) (n43) {};
		\node at (4,2) (n42) {};
		\node at (4,1) (n41) {};

		\node at (5,5) (n55) {};		
		\node at (5,4) (n54) {};
		\node at (5,3) (n53) {};
		\node at (5,2) (n52) {};
		\node at (5,1) (n51) {};
		
		\node at (6,5) (n65) {};		
		\node at (6,4) (n64) {};
		\node at (6,3) (n63) {};
		\node at (6,2) (n62) {};
		\node at (6,1) (n61) {};

		\node at (7,5) (n75) {};		
		\node at (7,4) (n74) {};
		\node at (7,3) (n73) {};
		\node at (7,2) (n72) {};
		\node at (7,1) (n71) {};

		
		\node[roundnode] at (8,5) (n85) {};
		\node[roundnode] at (8,4) (n84) {};
		\node[roundnode] at (8,3) (n83) {};
		\node[roundnode] at (8,2) (n82) {};
		\node[roundnode] at (8,1) (n81) {};
		
		\draw[very thick] (n15) -- (n25.center);
		\draw[very thick] (n14) -- (n23.center);
		\draw[very thick] (n13) -- (n24.center);
		\draw[very thick] (n12) -- (n22.center);
        \draw[very thick] (n11) -- (n51.center);		
        
        \draw[very thick] (n25.center) -- (n35.center);
        \draw[very thick] (n24.center) -- (n34.center);
        \draw[very thick] (n23.center) -- (n32.center);
        \draw[very thick] (n22.center) -- (n33.center);
        
        \draw[very thick] (n35.center) -- (n44.center);
        \draw[very thick] (n34.center) -- (n45.center);
        \draw[very thick] (n33.center) -- (n43.center);
        \draw[very thick] (n32.center) -- (n52.center);
		
        \draw[very thick] (n45.center) -- (n85.center);
        \draw[very thick] (n44.center) -- (n53.center);
        \draw[very thick] (n43.center) -- (n54.center);
        \draw[very thick] (n42.center) -- (n52.center);

        \draw[very thick] (n54.center) -- (n74.center);
        
        \draw[very thick] (n53.center) -- (n63.center);
        
        \draw[very thick] (n51.center) -- (n62.center);	
        
        \draw[very thick] (n52.center) -- (n61.center);

        
        \draw[very thick] (n64.center) -- (n74.center);
        
        \draw[very thick] (n63.center) -- (n72.center);			
        
        \draw[very thick] (n62.center) -- (n73.center);
        
        \draw[very thick] (n61.center) -- (n81.center);

        
        \draw[very thick] (n74.center) -- (n83.center);	
        
        \draw[very thick] (n73.center) -- (n84.center);	        
        				
		\draw[very thick] (n72.center) -- (n82.center);

	\end{scope}
\end{tikzpicture}}
\caption{The permutation $\sigma \in S_5$.}
\label{fig:45132}
\end{figure}
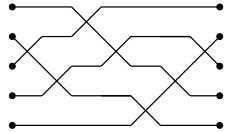	

As an example, take
$z_0 = \hat a_1 \acute \sigma$, where
\[ \acute\sigma = 
\frac{-1 + \hat a_1\hat a_2
+ \hat a_3 - \hat a_1\hat a_2\hat a_3
+ \hat a_1\hat a_4 + \hat a_2\hat a_4
-\hat a_1 \hat a_3 \hat a_4 - \hat a_2\hat a_3\hat a_4}{2\sqrt{2}}. \]

\begin{figure}[h!]
\centering

\resizebox{0.50\textwidth}{!}{

\begin{tikzpicture}[roundnode/.style={circle, draw=black, fill=black, minimum size=0.1mm,inner sep=2pt},
squarednode/.style={rectangle, draw=red!60, fill=red!5, very thick, minimum size=5mm},
diamondnode/.style={draw,diamond, fill=black, minimum size=1mm,very thick,inner sep=4pt},
diamondnodew/.style={draw,diamond, fill=white, minimum size=1mm,very thick,inner sep=4pt},
bigroundnode/.style={circle, draw=black, fill=black, minimum size=0.1mm,very thick, inner sep=4pt},
bigroundnodew/.style={circle, draw=black, fill=white, minimum size=0.1mm,very thick, inner sep=4pt},]


	\begin{scope}[shift={(-8.5,0)}]
		\draw[ultra thick,fill=white] (0.5,0.5) -- (0.5,5.5) -- (8.5,5.5) -- (8.5,0.5) -- cycle;
		
		\draw[ultra thick] (8.5,3) -- (10.5,3);
		\draw[ultra thick] (18.5,3) -- (21.5,3);
		\draw[ultra thick] (14.5,1) -- (14.5,-5);
		\node[roundnode] at (1,5) (n15) {};
		\node[roundnode] at (1,4) (n14) {};
		\node[roundnode] at (1,3) (n13) {};
		\node[roundnode] at (1,2) (n12) {};
		\node[roundnode] at (1,1) (n11) {};
		
        \node at (2,5) (n25) {};		
		\node at (2,4) (n24) {};
		\node at (2,3) (n23) {};
		\node at (2,2) (n22) {};
		\node at (2,1) (n21) {};

		\node at (3,5) (n35) {};
		\node at (3,4) (n34) {};
		\node at (3,3) (n33) {};
		\node at (3,2) (n32) {};
		\node at (3,1) (n31) {};
		
        \node at (4,5) (n45) {};		
		\node at (4,4) (n44) {};
		\node at (4,3) (n43) {};
		\node at (4,2) (n42) {};
		\node at (4,1) (n41) {};

		\node at (5,5) (n55) {};		
		\node at (5,4) (n54) {};
		\node at (5,3) (n53) {};
		\node at (5,2) (n52) {};
		\node at (5,1) (n51) {};
		
		\node at (6,5) (n65) {};		
		\node at (6,4) (n64) {};
		\node at (6,3) (n63) {};
		\node at (6,2) (n62) {};
		\node at (6,1) (n61) {};

		\node at (7,5) (n75) {};		
		\node at (7,4) (n74) {};
		\node at (7,3) (n73) {};
		\node at (7,2) (n72) {};
		\node at (7,1) (n71) {};

		
		\node[roundnode] at (8,5) (n85) {};
		\node[roundnode] at (8,4) (n84) {};
		\node[roundnode] at (8,3) (n83) {};
		\node[roundnode] at (8,2) (n82) {};
		\node[roundnode] at (8,1) (n81) {};
		
		\draw[very thick] (n15) -- (n35.center);
		\draw[very thick] (n14) -- (n23.center);
		\draw[very thick] (n13) -- (n24.center);
		\draw[very thick] (n12) -- (n22.center);
        \draw[very thick] (n11) -- (n51.center);		
        
        \draw[very thick] (n24.center) -- (n34.center);
        \draw[very thick] (n23.center) -- (n32.center);
        \draw[very thick] (n22.center) -- (n33.center);
        
        \draw[very thick] (n35.center) -- (n44.center);
        \draw[very thick] (n34.center) -- (n45.center);
        \draw[very thick] (n33.center) -- (n43.center);
        \draw[very thick] (n32.center) -- (n52.center);
		
        \draw[very thick] (n45.center) -- (n85.center);
        \draw[very thick] (n44.center) -- (n53.center);
        \draw[very thick] (n43.center) -- (n54.center);
        \draw[very thick] (n42.center) -- (n52.center);

        \draw[very thick] (n54.center) -- (n74.center);
        
        \draw[very thick] (n53.center) -- (n63.center);
        
        \draw[very thick] (n51.center) -- (n62.center);	
        
        \draw[very thick] (n52.center) -- (n61.center);

        
        \draw[very thick] (n64.center) -- (n74.center);
        
        \draw[very thick] (n63.center) -- (n72.center);			
        
        \draw[very thick] (n62.center) -- (n73.center);
        
        \draw[very thick] (n61.center) -- (n81.center);

        
        \draw[very thick] (n74.center) -- (n83.center);	
        
        \draw[very thick] (n73.center) -- (n84.center);	        
        				
		\draw[very thick] (n72.center) -- (n82.center);	
		
		\node[bigroundnodew] at (1.5,3.5) (m1) {};
		\node[bigroundnode] at (3.5,4.5) (m2) {};
		\node[bigroundnodew] at (2.5,2.5) (m3) {};
		\node[bigroundnode] at (4.5,3.5) (m4) {};
		\node[bigroundnode] at (5.5,1.5) (m5) {};
		\node[bigroundnodew] at (6.5,2.5) (m6) {};
		\node[bigroundnode] at (7.5,3.5) (m7) {};
		
	\end{scope}

	\begin{scope}[shift={(1.5,0)}]
		\draw[ultra thick,fill=white] (0.5,0.5) -- (0.5,5.5) -- (8.5,5.5) -- (8.5,0.5) -- cycle;
		\node[roundnode] at (1,5) (n15) {};
		\node[roundnode] at (1,4) (n14) {};
		\node[roundnode] at (1,3) (n13) {};
		\node[roundnode] at (1,2) (n12) {};
		\node[roundnode] at (1,1) (n11) {};
		
        \node at (2,5) (n25) {};		
		\node at (2,4) (n24) {};
		\node at (2,3) (n23) {};
		\node at (2,2) (n22) {};
		\node at (2,1) (n21) {};

		\node at (3,5) (n35) {};
		\node at (3,4) (n34) {};
		\node at (3,3) (n33) {};
		\node at (3,2) (n32) {};
		\node at (3,1) (n31) {};
		
        \node at (4,5) (n45) {};		
		\node at (4,4) (n44) {};
		\node at (4,3) (n43) {};
		\node at (4,2) (n42) {};
		\node at (4,1) (n41) {};

		\node at (5,5) (n55) {};		
		\node at (5,4) (n54) {};
		\node at (5,3) (n53) {};
		\node at (5,2) (n52) {};
		\node at (5,1) (n51) {};
		
		\node at (6,5) (n65) {};		
		\node at (6,4) (n64) {};
		\node at (6,3) (n63) {};
		\node at (6,2) (n62) {};
		\node at (6,1) (n61) {};

		\node at (7,5) (n75) {};		
		\node at (7,4) (n74) {};
		\node at (7,3) (n73) {};
		\node at (7,2) (n72) {};
		\node at (7,1) (n71) {};

		
		\node[roundnode] at (8,5) (n85) {};
		\node[roundnode] at (8,4) (n84) {};
		\node[roundnode] at (8,3) (n83) {};
		\node[roundnode] at (8,2) (n82) {};
		\node[roundnode] at (8,1) (n81) {};
		
		\draw[very thick] (n15) -- (n35.center);
		\draw[very thick] (n14) -- (n23.center);
		\draw[very thick] (n13) -- (n24.center);
		\draw[very thick] (n12) -- (n22.center);
        \draw[very thick] (n11) -- (n51.center);		
        
        \draw[very thick] (n24.center) -- (n34.center);
        \draw[very thick] (n23.center) -- (n32.center);
        \draw[very thick] (n22.center) -- (n33.center);
        
        \draw[very thick] (n35.center) -- (n44.center);
        \draw[very thick] (n34.center) -- (n45.center);
        \draw[very thick] (n33.center) -- (n43.center);
        \draw[very thick] (n32.center) -- (n52.center);
		
        \draw[very thick] (n45.center) -- (n85.center);
        \draw[very thick] (n44.center) -- (n53.center);
        \draw[very thick] (n43.center) -- (n54.center);
        \draw[very thick] (n42.center) -- (n52.center);

        \draw[very thick] (n54.center) -- (n74.center);
        
        \draw[very thick] (n53.center) -- (n63.center);
        
        \draw[very thick] (n51.center) -- (n62.center);	
        
        \draw[very thick] (n52.center) -- (n61.center);

        
        \draw[very thick] (n64.center) -- (n74.center);
        
        \draw[very thick] (n63.center) -- (n72.center);			
        
        \draw[very thick] (n62.center) -- (n73.center);
        
        \draw[very thick] (n61.center) -- (n81.center);

        
        \draw[very thick] (n74.center) -- (n83.center);	
        
        \draw[very thick] (n73.center) -- (n84.center);	        
        				
		\draw[very thick] (n72.center) -- (n82.center);	
		
		\node[bigroundnode] at (1.5,3.5) (m1) {};
		\node[bigroundnodew] at (3.5,4.5) (m2) {};
		\node[bigroundnode] at (2.5,2.5) (m3) {};
		\node[bigroundnodew] at (4.5,3.5) (m4) {};
		\node[bigroundnode] at (5.5,1.5) (m5) {};
		\node[bigroundnodew] at (6.5,2.5) (m6) {};
		\node[bigroundnode] at (7.5,3.5) (m7) {};
		
	\end{scope}

	\begin{scope}[shift={(1.5,-7)}]
		\draw[ultra thick,fill=white] (0.5,0.5) -- (0.5,5.5) -- (8.5,5.5) -- (8.5,0.5) -- cycle;
		\node[roundnode] at (1,5) (n15) {};
		\node[roundnode] at (1,4) (n14) {};
		\node[roundnode] at (1,3) (n13) {};
		\node[roundnode] at (1,2) (n12) {};
		\node[roundnode] at (1,1) (n11) {};
		
        \node at (2,5) (n25) {};		
		\node at (2,4) (n24) {};
		\node at (2,3) (n23) {};
		\node at (2,2) (n22) {};
		\node at (2,1) (n21) {};

		\node at (3,5) (n35) {};
		\node at (3,4) (n34) {};
		\node at (3,3) (n33) {};
		\node at (3,2) (n32) {};
		\node at (3,1) (n31) {};
		
        \node at (4,5) (n45) {};		
		\node at (4,4) (n44) {};
		\node at (4,3) (n43) {};
		\node at (4,2) (n42) {};
		\node at (4,1) (n41) {};

		\node at (5,5) (n55) {};		
		\node at (5,4) (n54) {};
		\node at (5,3) (n53) {};
		\node at (5,2) (n52) {};
		\node at (5,1) (n51) {};
		
		\node at (6,5) (n65) {};		
		\node at (6,4) (n64) {};
		\node at (6,3) (n63) {};
		\node at (6,2) (n62) {};
		\node at (6,1) (n61) {};

		\node at (7,5) (n75) {};		
		\node at (7,4) (n74) {};
		\node at (7,3) (n73) {};
		\node at (7,2) (n72) {};
		\node at (7,1) (n71) {};

		
		\node[roundnode] at (8,5) (n85) {};
		\node[roundnode] at (8,4) (n84) {};
		\node[roundnode] at (8,3) (n83) {};
		\node[roundnode] at (8,2) (n82) {};
		\node[roundnode] at (8,1) (n81) {};
		
		\draw[very thick] (n15) -- (n35.center);
		\draw[very thick] (n14) -- (n23.center);
		\draw[very thick] (n13) -- (n24.center);
		\draw[very thick] (n12) -- (n22.center);
        \draw[very thick] (n11) -- (n51.center);		
        
        \draw[very thick] (n24.center) -- (n34.center);
        \draw[very thick] (n23.center) -- (n32.center);
        \draw[very thick] (n22.center) -- (n33.center);
        
        \draw[very thick] (n35.center) -- (n44.center);
        \draw[very thick] (n34.center) -- (n45.center);
        \draw[very thick] (n33.center) -- (n43.center);
        \draw[very thick] (n32.center) -- (n52.center);
		
        \draw[very thick] (n45.center) -- (n85.center);
        \draw[very thick] (n44.center) -- (n53.center);
        \draw[very thick] (n43.center) -- (n54.center);
        \draw[very thick] (n42.center) -- (n52.center);

        \draw[very thick] (n54.center) -- (n74.center);
        
        \draw[very thick] (n53.center) -- (n63.center);
        
        \draw[very thick] (n51.center) -- (n62.center);	
        
        \draw[very thick] (n52.center) -- (n61.center);

        
        \draw[very thick] (n64.center) -- (n74.center);
        
        \draw[very thick] (n63.center) -- (n72.center);			
        
        \draw[very thick] (n62.center) -- (n73.center);
        
        \draw[very thick] (n61.center) -- (n81.center);

        
        \draw[very thick] (n74.center) -- (n83.center);	
        
        \draw[very thick] (n73.center) -- (n84.center);	        
        				
		\draw[very thick] (n72.center) -- (n82.center);	
		
		\node[bigroundnode] at (1.5,3.5) (m1) {};
		\node[bigroundnodew] at (3.5,4.5) (m2) {};
		\node[bigroundnodew] at (2.5,2.5) (m3) {};
		\node[bigroundnode] at (4.5,3.5) (m4) {};
		\node[bigroundnodew] at (5.5,1.5) (m5) {};
		\node[bigroundnode] at (6.5,2.5) (m6) {};
		\node[bigroundnode] at (7.5,3.5) (m7) {};
		
	\end{scope}

	\begin{scope}[shift={(11.5,0)}]
		\draw[ultra thick,fill=white] (0.5,0.5) -- (0.5,5.5) -- (8.5,5.5) -- (8.5,0.5) -- cycle;
		\node[roundnode] at (1,5) (n15) {};
		\node[roundnode] at (1,4) (n14) {};
		\node[roundnode] at (1,3) (n13) {};
		\node[roundnode] at (1,2) (n12) {};
		\node[roundnode] at (1,1) (n11) {};
		
        \node at (2,5) (n25) {};		
		\node at (2,4) (n24) {};
		\node at (2,3) (n23) {};
		\node at (2,2) (n22) {};
		\node at (2,1) (n21) {};

		\node at (3,5) (n35) {};
		\node at (3,4) (n34) {};
		\node at (3,3) (n33) {};
		\node at (3,2) (n32) {};
		\node at (3,1) (n31) {};
		
        \node at (4,5) (n45) {};		
		\node at (4,4) (n44) {};
		\node at (4,3) (n43) {};
		\node at (4,2) (n42) {};
		\node at (4,1) (n41) {};

		\node at (5,5) (n55) {};		
		\node at (5,4) (n54) {};
		\node at (5,3) (n53) {};
		\node at (5,2) (n52) {};
		\node at (5,1) (n51) {};
		
		\node at (6,5) (n65) {};		
		\node at (6,4) (n64) {};
		\node at (6,3) (n63) {};
		\node at (6,2) (n62) {};
		\node at (6,1) (n61) {};

		\node at (7,5) (n75) {};		
		\node at (7,4) (n74) {};
		\node at (7,3) (n73) {};
		\node at (7,2) (n72) {};
		\node at (7,1) (n71) {};

		
		\node[roundnode] at (8,5) (n85) {};
		\node[roundnode] at (8,4) (n84) {};
		\node[roundnode] at (8,3) (n83) {};
		\node[roundnode] at (8,2) (n82) {};
		\node[roundnode] at (8,1) (n81) {};
		
\draw[very thick] (n15) -- (n35.center);
		\draw[very thick] (n14) -- (n23.center);
		\draw[very thick] (n13) -- (n24.center);
		\draw[very thick] (n12) -- (n22.center);
        \draw[very thick] (n11) -- (n51.center);		
        
        \draw[very thick] (n24.center) -- (n34.center);
        \draw[very thick] (n23.center) -- (n32.center);
        \draw[very thick] (n22.center) -- (n33.center);
        
        \draw[very thick] (n35.center) -- (n44.center);
        \draw[very thick] (n34.center) -- (n45.center);
        \draw[very thick] (n33.center) -- (n43.center);
        \draw[very thick] (n32.center) -- (n52.center);
		
        \draw[very thick] (n45.center) -- (n85.center);
        \draw[very thick] (n44.center) -- (n53.center);
        \draw[very thick] (n43.center) -- (n54.center);
        \draw[very thick] (n42.center) -- (n52.center);

        \draw[very thick] (n54.center) -- (n74.center);
        
        \draw[very thick] (n53.center) -- (n63.center);
        
        \draw[very thick] (n51.center) -- (n62.center);	
        
        \draw[very thick] (n52.center) -- (n61.center);

        
        \draw[very thick] (n64.center) -- (n74.center);
        
        \draw[very thick] (n63.center) -- (n72.center);			
        
        \draw[very thick] (n62.center) -- (n73.center);
        
        \draw[very thick] (n61.center) -- (n81.center);

        
        \draw[very thick] (n74.center) -- (n83.center);	
        
        \draw[very thick] (n73.center) -- (n84.center);	        
        				
		\draw[very thick] (n72.center) -- (n82.center);	
		
		\node[bigroundnode] at (1.5,3.5) (m1) {};
		\node[bigroundnodew] at (3.5,4.5) (m2) {};
		\node[bigroundnode] at (2.5,2.5) (m3) {};
		\node[bigroundnode] at (4.5,3.5) (m4) {};
		\node[bigroundnode] at (5.5,1.5) (m5) {};
		\node[bigroundnode] at (6.5,2.5) (m6) {};
		\node[bigroundnodew] at (7.5,3.5) (m7) {};
		
	\end{scope}

	\end{tikzpicture}}
	
	\caption{The stratification of $\BL_{z_0}$.}
\label{fig:45132x}
\end{figure}

\newpage

A case by case verification shows that $\BL_{z_0}$
has $4$ strata of dimension $0$,
$3$ strata of dimension $1$ and
no strata of dimension higher than $1$.
It follows that $\BL_{z_0}$ is homotopically equivalent
to the graph in Figure~\ref{fig:45132x}
and therefore contractible.
In the figure, black indicates $\varepsilon(k) = -1$
and white indicates $\varepsilon(k) = +1$.
\end{example}


\begin{example}
\label{section:n3}

Take $n=3$ and $\sigma = [4 3 1 2]$. 
Let us fix the reduced word $a_1a_2a_3a_1a_2$, 
shown in the Figure~\ref{fig:4312}.

\begin{figure}[ht!]
\centering
\resizebox{0.2\textwidth}{!}{

\begin{tikzpicture}[roundnode/.style={circle, draw=black, fill=black, minimum size=0.1mm,inner sep=2pt},
squarednode/.style={rectangle, draw=red!60, fill=red!5, very thick, minimum size=5mm},
diamondnode/.style={draw,diamond, fill=black, minimum size=1mm,very thick,inner sep=4pt},
diamondnodew/.style={draw,diamond, fill=white, minimum size=1mm,very thick,inner sep=4pt},
bigdiamondnodew/.style={draw,diamond, fill=white, minimum size=1mm,very thick, inner sep=4pt},
bigroundnode/.style={circle, draw=black, fill=black, minimum size=0.1mm,very thick, inner sep=4pt},
bigroundnodew/.style={circle, draw=black, fill=white, minimum size=0.1mm,very thick, inner sep=4pt},]

	
	\begin{scope}[shift={(-5.5,0)}]
		
		\node[roundnode] at (1,4) (n14) {};
		\node[roundnode] at (1,3) (n13) {};
		\node[roundnode] at (1,2) (n12) {};
		\node[roundnode] at (1,1) (n11) {};
		
       		
		\node at (2,4) (n24) {};
		\node at (2,3) (n23) {};
		\node at (2,2) (n22) {};
		\node at (2,1) (n21) {};

		\node at (3,4) (n34) {};
		\node at (3,3) (n33) {};
		\node at (3,2) (n32) {};
		\node at (3,1) (n31) {};

		\node at (4,4) (n44) {};
		\node at (4,3) (n43) {};
		\node at (4,2) (n42) {};
		\node at (4,1) (n41) {};

		\node at (5,4) (n54) {};
		\node at (5,3) (n53) {};
		\node at (5,2) (n52) {};
		\node at (5,1) (n51) {};

		\node at (6,4) (n64) {};
		\node at (6,3) (n63) {};
		\node at (6,2) (n62) {};
		\node at (6,1) (n61) {};	
	
		
		\node[roundnode] at (6,4) (n64) {};
		\node[roundnode] at (6,3) (n63) {};
		\node[roundnode] at (6,2) (n62) {};
		\node[roundnode] at (6,1) (n61) {};
		
		\draw[very thick] (n14) -- (n23.center);
		\draw[very thick] (n13) -- (n24.center);
		\draw[very thick] (n12) -- (n22.center);
        \draw[very thick] (n11) -- (n31.center);		
         
        \draw[very thick] (n24.center) -- (n44.center);
        \draw[very thick] (n23.center) -- (n32.center);
        \draw[very thick] (n22.center) -- (n33.center);
        
        \draw[very thick] (n33.center) -- (n43.center);
        \draw[very thick] (n32.center) -- (n41.center);
        \draw[very thick] (n31.center) -- (n42.center);
		
        \draw[very thick] (n44.center) -- (n53.center);
        \draw[very thick] (n43.center) -- (n54.center);
        \draw[very thick] (n42.center) -- (n52.center);
        \draw[very thick] (n41.center) -- (n61.center);

        \draw[very thick] (n54.center) -- (n64.center);
        \draw[very thick] (n52.center) -- (n63.center);
        \draw[very thick] (n53.center) -- (n62.center);

	\end{scope}
\end{tikzpicture}}
\caption{The permutation $\sigma \in S_4$.}
\label{fig:4312}
\end{figure}

Write $L \in \Lo_4^1$ as
 \[ L = \begin{pmatrix}
1 & & & \\ x & 1 & & \\
u & y & 1 & \\ w & v & z & 1 
\end{pmatrix}, \, u, v, w, x, y, z \in \mathbb{R}. \]

Applying the definition, the set $\BL_\sigma$ is

\[ \BL_\sigma = \bigg\{ L \, | \, w=0, \, u \neq 0, v \neq 0,  xyz-xv-uz \neq 0  \bigg\} \subset \Lo^1_4.\] 

A computations shows that $\acute{\sigma}=\frac{\sqrt[]{2}}{4}(-1+ \hat{a}_1 +\hat{a}_2 + \hat{a}_1 \hat{a}_2 + \hat{a}_3 + \hat{a}_1\hat{a}_3+ \hat{a}_2\hat{a}_3-\hat{a}_1\hat{a}_2\hat{a}_3)$: notice that $|-\frac{\sqrt[]{2}}{4}|=2^{-(n+1-c)/2}$.
We have $\ell = 5$, $b = |\Block(\sigma)| = 0$, $c = \nc(\sigma) = 1$ and
if $z \in \acute{\sigma} \Quat_{4}$ implies $\Re(z)$
equal $-\dfrac{\sqrt[]{2}}{4}$ or  $\dfrac{\sqrt[]{2}}{4}$.
Therefore and from Remark~\ref{remark_orbits}, for all $z \in \acute{\sigma} \Quat_{4}$ the orbit $\cO(z)$ has size $2^{n-c+1}=8$.
Now we describe $\BL_z$, for $z \in \acute{\sigma} \Quat_{4}$.

The matrices $L \in \BLS_\varepsilon$,
where $\varepsilon$ is an ancestry of dimension $0$, can be written as the following product
\[ L = \lambda_1(t_1) \lambda_2(t_2) \lambda_3(t_3) \lambda_1(t_4)\lambda_2(t_5) = 
\begin{pmatrix}
1 & & & \\ t_1+t_4 & 1 & & \\
t_2 t_4 & t_2 + t_5 & 1 & \\ 0 & t_3t_5 & t_3 & 1 
\end{pmatrix} \]
then $u = t_2 t_4, \, v = t_3 t_5, \, xyz-xv-uz = t_1 t_2 t_3$.

Set $z_1=\acute{a}_1 \acute{a}_2 \grave{a}_3 \acute{a}_1 \grave{a}_2=\frac{\sqrt[]{2}}{4}(1+ \hat{a}_1 +\hat{a}_2 - \hat{a}_1 \hat{a}_2 + \hat{a}_3 - \hat{a}_1\hat{a}_3- \hat{a}_2\hat{a}_3-\hat{a}_1\hat{a}_2\hat{a}_3)$.
From Equation~\ref{theo:N}, $N(z_1)=2^{5-3+0-1}+2^{5/2-1}(\frac{\sqrt[]{2}}{4})=3$,
therefore the set $\BL_{z_1}$ has $3$ ancestries of dimension $0$.
Moreover, 
$\BL_{z_1}$ has 
$2$ ancestries of dimension $1$ and
no ancestries of dimension higher than $1$. 
The characterization of the set $\BL_{z_1}$ is given by
\[ \BL_{z_1}=\bigg\{ L \, | \, w=0, \, u > 0, v > 0,  xyz-xv-uz > 0  \bigg\}. \]
Therefore the set $\BL_{z_1}$ contains three open strata
\[ \BL_{(-1,-1,+1,-1,+1)} = \{ L \, | \,  z > 0, \, yz - v < 0 \}, \]
\[ \BL_{(-1,+1,+1,-1,-1)} = \{ L \, | \,z < 0, \, yz - v < 0 \}, \]
\[ \BL_{(+1,-1,-1,-1,-1)} = \{ L \, | \, z < 0, \, yz - v > 0 \}, \]
and two strata of dimension $1$ 
\[ \BL_{(-1,-2,+1,-1,+2)} = \{ L \, | \, z = 0, \, yz - v < 0 \}, \]
\[ \BL_{(-2,+1,-1,+2,-1)} = \{ L \, | \, z < 0, \, yz - v = 0 \}. \]
Finally, the set $\BL_{z_1}$ is homotopically equivalent to the CW complex in Figure~\ref{fig:4312xx}.
It follows that $\BL_{z_1}$ is contractible.

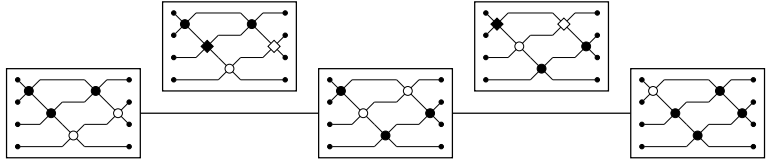
\begin{figure}[ht!]
\centering
\resizebox{0.7\textwidth}{!}{
\begin{tikzpicture}[roundnode/.style={circle, draw=black, fill=black, minimum size=0.1mm,inner sep=2pt},
squarednode/.style={rectangle, draw=red!60, fill=red!5, very thick, minimum size=5mm},
diamondnode/.style={draw,diamond, fill=black, minimum size=1mm,very thick,inner sep=4pt},
diamondnodew/.style={draw,diamond, fill=white, minimum size=1mm,very thick,inner sep=4pt},
bigdiamondnodew/.style={draw,diamond, fill=white, minimum size=1mm,very thick, inner sep=4pt},
bigdiamondnode/.style={draw,diamond, fill=black, minimum size=1mm,very thick, inner sep=4pt},
bigroundnode/.style={circle, draw=black, fill=black, minimum size=0.1mm,very thick, inner sep=4pt},
bigroundnodew/.style={circle, draw=black, fill=white, minimum size=0.1mm,very thick, inner sep=4pt},]

	\draw[ultra thick] (-19.5,2.5) -- (6.5,2.5);
	
	

	\begin{scope}[shift={(-24.5,0)}]
		\draw[ultra thick,fill=white] (0.5,0.5) -- (0.5,4.5) -- (6.5,4.5) -- (6.5,0.5) -- cycle;
		
		\node[roundnode] at (1,4) (n14) {};
		\node[roundnode] at (1,3) (n13) {};
		\node[roundnode] at (1,2) (n12) {};
		\node[roundnode] at (1,1) (n11) {};
		
       		
		\node at (2,4) (n24) {};
		\node at (2,3) (n23) {};
		\node at (2,2) (n22) {};
		\node at (2,1) (n21) {};

		\node at (3,4) (n34) {};
		\node at (3,3) (n33) {};
		\node at (3,2) (n32) {};
		\node at (3,1) (n31) {};

		\node at (4,4) (n44) {};
		\node at (4,3) (n43) {};
		\node at (4,2) (n42) {};
		\node at (4,1) (n41) {};

		\node at (5,4) (n54) {};
		\node at (5,3) (n53) {};
		\node at (5,2) (n52) {};
		\node at (5,1) (n51) {};

		\node at (6,4) (n64) {};
		\node at (6,3) (n63) {};
		\node at (6,2) (n62) {};
		\node at (6,1) (n61) {};	
	
		
		\node[roundnode] at (6,4) (n64) {};
		\node[roundnode] at (6,3) (n63) {};
		\node[roundnode] at (6,2) (n62) {};
		\node[roundnode] at (6,1) (n61) {};
		
		\draw[very thick] (n14) -- (n23.center);
		\draw[very thick] (n13) -- (n24.center);
		\draw[very thick] (n12) -- (n22.center);
        \draw[very thick] (n11) -- (n31.center);		
         
        \draw[very thick] (n24.center) -- (n44.center);
        \draw[very thick] (n23.center) -- (n32.center);
        \draw[very thick] (n22.center) -- (n33.center);
        
        \draw[very thick] (n33.center) -- (n43.center);
        \draw[very thick] (n32.center) -- (n41.center);
        \draw[very thick] (n31.center) -- (n42.center);
		
        \draw[very thick] (n44.center) -- (n53.center);
        \draw[very thick] (n43.center) -- (n54.center);
        \draw[very thick] (n42.center) -- (n52.center);
        \draw[very thick] (n41.center) -- (n61.center);

        \draw[very thick] (n54.center) -- (n64.center);
        \draw[very thick] (n52.center) -- (n63.center);
        \draw[very thick] (n53.center) -- (n62.center);	
		
		\node[bigroundnode] at (1.5,3.5) (m1) {};
		\node[bigroundnode] at (2.5,2.5) (m2) {};
		\node[bigroundnodew] at (3.5,1.5) (m3) {};
		\node[bigroundnode] at (4.5,3.5) (m4) {};
		\node[bigroundnodew] at (5.5,2.5) (m5) {};
	
	\end{scope}


	\begin{scope}[shift={(-17.5,3)}]
		\draw[ultra thick,fill=white] (0.5,0.5) -- (0.5,4.5) -- (6.5,4.5) -- (6.5,0.5) -- cycle;
		
		
		\node[roundnode] at (1,4) (n14) {};
		\node[roundnode] at (1,3) (n13) {};
		\node[roundnode] at (1,2) (n12) {};
		\node[roundnode] at (1,1) (n11) {};
		
       		
		\node at (2,4) (n24) {};
		\node at (2,3) (n23) {};
		\node at (2,2) (n22) {};
		\node at (2,1) (n21) {};

		\node at (3,4) (n34) {};
		\node at (3,3) (n33) {};
		\node at (3,2) (n32) {};
		\node at (3,1) (n31) {};

		\node at (4,4) (n44) {};
		\node at (4,3) (n43) {};
		\node at (4,2) (n42) {};
		\node at (4,1) (n41) {};

		\node at (5,4) (n54) {};
		\node at (5,3) (n53) {};
		\node at (5,2) (n52) {};
		\node at (5,1) (n51) {};

		\node at (6,4) (n64) {};
		\node at (6,3) (n63) {};
		\node at (6,2) (n62) {};
		\node at (6,1) (n61) {};	
	
		
		\node[roundnode] at (6,4) (n64) {};
		\node[roundnode] at (6,3) (n63) {};
		\node[roundnode] at (6,2) (n62) {};
		\node[roundnode] at (6,1) (n61) {};
		
		\draw[very thick] (n14) -- (n23.center);
		\draw[very thick] (n13) -- (n24.center);
		\draw[very thick] (n12) -- (n22.center);
        \draw[very thick] (n11) -- (n31.center);		
         
        \draw[very thick] (n24.center) -- (n44.center);
        \draw[very thick] (n23.center) -- (n32.center);
        \draw[very thick] (n22.center) -- (n33.center);
        
        \draw[very thick] (n33.center) -- (n43.center);
        \draw[very thick] (n32.center) -- (n41.center);
        \draw[very thick] (n31.center) -- (n42.center);
		
        \draw[very thick] (n44.center) -- (n53.center);
        \draw[very thick] (n43.center) -- (n54.center);
        \draw[very thick] (n42.center) -- (n52.center);
        \draw[very thick] (n41.center) -- (n61.center);

        \draw[very thick] (n54.center) -- (n64.center);
        \draw[very thick] (n52.center) -- (n63.center);
        \draw[very thick] (n53.center) -- (n62.center);	
		
		\node[bigroundnode] at (1.5,3.5) (m1) {};
		\node[bigdiamondnode] at (2.5,2.5) (m2) {};
		\node[bigroundnodew] at (3.5,1.5) (m3) {};
		\node[bigroundnode] at (4.5,3.5) (m4) {};
		\node[bigdiamondnodew] at (5.5,2.5) (m5) {};

	\end{scope}
	
	\begin{scope}[shift={(-10.5,0)}]
		\draw[ultra thick,fill=white] (0.5,0.5) -- (0.5,4.5) -- (6.5,4.5) -- (6.5,0.5) -- cycle;
		
		
		\node[roundnode] at (1,4) (n14) {};
		\node[roundnode] at (1,3) (n13) {};
		\node[roundnode] at (1,2) (n12) {};
		\node[roundnode] at (1,1) (n11) {};
		
       		
		\node at (2,4) (n24) {};
		\node at (2,3) (n23) {};
		\node at (2,2) (n22) {};
		\node at (2,1) (n21) {};

		\node at (3,4) (n34) {};
		\node at (3,3) (n33) {};
		\node at (3,2) (n32) {};
		\node at (3,1) (n31) {};

		\node at (4,4) (n44) {};
		\node at (4,3) (n43) {};
		\node at (4,2) (n42) {};
		\node at (4,1) (n41) {};

		\node at (5,4) (n54) {};
		\node at (5,3) (n53) {};
		\node at (5,2) (n52) {};
		\node at (5,1) (n51) {};

		\node at (6,4) (n64) {};
		\node at (6,3) (n63) {};
		\node at (6,2) (n62) {};
		\node at (6,1) (n61) {};	
	
		
		\node[roundnode] at (6,4) (n64) {};
		\node[roundnode] at (6,3) (n63) {};
		\node[roundnode] at (6,2) (n62) {};
		\node[roundnode] at (6,1) (n61) {};
		
		\draw[very thick] (n14) -- (n23.center);
		\draw[very thick] (n13) -- (n24.center);
		\draw[very thick] (n12) -- (n22.center);
        \draw[very thick] (n11) -- (n31.center);		
         
        \draw[very thick] (n24.center) -- (n44.center);
        \draw[very thick] (n23.center) -- (n32.center);
        \draw[very thick] (n22.center) -- (n33.center);
        
        \draw[very thick] (n33.center) -- (n43.center);
        \draw[very thick] (n32.center) -- (n41.center);
        \draw[very thick] (n31.center) -- (n42.center);
		
        \draw[very thick] (n44.center) -- (n53.center);
        \draw[very thick] (n43.center) -- (n54.center);
        \draw[very thick] (n42.center) -- (n52.center);
        \draw[very thick] (n41.center) -- (n61.center);

        \draw[very thick] (n54.center) -- (n64.center);
        \draw[very thick] (n52.center) -- (n63.center);
        \draw[very thick] (n53.center) -- (n62.center);	
		
		\node[bigroundnode] at (1.5,3.5) (m1) {};
		\node[bigroundnodew] at (2.5,2.5) (m2) {};
		\node[bigroundnode] at (3.5,1.5) (m3) {};
		\node[bigroundnodew] at (4.5,3.5) (m4) {};
		\node[bigroundnode] at (5.5,2.5) (m5) {};

	\end{scope}
	
	

	\begin{scope}[shift={(-3.5,3)}]
		\draw[ultra thick,fill=white] (0.5,0.5) -- (0.5,4.5) -- (6.5,4.5) -- (6.5,0.5) -- cycle;
		
		
		\node[roundnode] at (1,4) (n14) {};
		\node[roundnode] at (1,3) (n13) {};
		\node[roundnode] at (1,2) (n12) {};
		\node[roundnode] at (1,1) (n11) {};
		
       		
		\node at (2,4) (n24) {};
		\node at (2,3) (n23) {};
		\node at (2,2) (n22) {};
		\node at (2,1) (n21) {};

		\node at (3,4) (n34) {};
		\node at (3,3) (n33) {};
		\node at (3,2) (n32) {};
		\node at (3,1) (n31) {};

		\node at (4,4) (n44) {};
		\node at (4,3) (n43) {};
		\node at (4,2) (n42) {};
		\node at (4,1) (n41) {};

		\node at (5,4) (n54) {};
		\node at (5,3) (n53) {};
		\node at (5,2) (n52) {};
		\node at (5,1) (n51) {};

		\node at (6,4) (n64) {};
		\node at (6,3) (n63) {};
		\node at (6,2) (n62) {};
		\node at (6,1) (n61) {};	
	
		
		\node[roundnode] at (6,4) (n64) {};
		\node[roundnode] at (6,3) (n63) {};
		\node[roundnode] at (6,2) (n62) {};
		\node[roundnode] at (6,1) (n61) {};
		
		\draw[very thick] (n14) -- (n23.center);
		\draw[very thick] (n13) -- (n24.center);
		\draw[very thick] (n12) -- (n22.center);
        \draw[very thick] (n11) -- (n31.center);		
         
        \draw[very thick] (n24.center) -- (n44.center);
        \draw[very thick] (n23.center) -- (n32.center);
        \draw[very thick] (n22.center) -- (n33.center);
        
        \draw[very thick] (n33.center) -- (n43.center);
        \draw[very thick] (n32.center) -- (n41.center);
        \draw[very thick] (n31.center) -- (n42.center);
		
        \draw[very thick] (n44.center) -- (n53.center);
        \draw[very thick] (n43.center) -- (n54.center);
        \draw[very thick] (n42.center) -- (n52.center);
        \draw[very thick] (n41.center) -- (n61.center);

        \draw[very thick] (n54.center) -- (n64.center);
        \draw[very thick] (n52.center) -- (n63.center);
        \draw[very thick] (n53.center) -- (n62.center);	
		
		\node[bigdiamondnode] at (1.5,3.5) (m1) {};
		\node[bigroundnodew] at (2.5,2.5) (m2) {};
		\node[bigroundnode] at (3.5,1.5) (m3) {};
		\node[bigdiamondnodew] at (4.5,3.5) (m4) {};
		\node[bigroundnode] at (5.5,2.5) (m5) {};

	\end{scope}	
	
	\begin{scope}[shift={(3.5,0)}]
		\draw[ultra thick,fill=white] (0.5,0.5) -- (0.5,4.5) -- (6.5,4.5) -- (6.5,0.5) -- cycle;
		
		
		\node[roundnode] at (1,4) (n14) {};
		\node[roundnode] at (1,3) (n13) {};
		\node[roundnode] at (1,2) (n12) {};
		\node[roundnode] at (1,1) (n11) {};
		
       		
		\node at (2,4) (n24) {};
		\node at (2,3) (n23) {};
		\node at (2,2) (n22) {};
		\node at (2,1) (n21) {};

		\node at (3,4) (n34) {};
		\node at (3,3) (n33) {};
		\node at (3,2) (n32) {};
		\node at (3,1) (n31) {};

		\node at (4,4) (n44) {};
		\node at (4,3) (n43) {};
		\node at (4,2) (n42) {};
		\node at (4,1) (n41) {};

		\node at (5,4) (n54) {};
		\node at (5,3) (n53) {};
		\node at (5,2) (n52) {};
		\node at (5,1) (n51) {};

		\node at (6,4) (n64) {};
		\node at (6,3) (n63) {};
		\node at (6,2) (n62) {};
		\node at (6,1) (n61) {};	
	
		
		\node[roundnode] at (6,4) (n64) {};
		\node[roundnode] at (6,3) (n63) {};
		\node[roundnode] at (6,2) (n62) {};
		\node[roundnode] at (6,1) (n61) {};
		
		\draw[very thick] (n14) -- (n23.center);
		\draw[very thick] (n13) -- (n24.center);
		\draw[very thick] (n12) -- (n22.center);
        \draw[very thick] (n11) -- (n31.center);		
         
        \draw[very thick] (n24.center) -- (n44.center);
        \draw[very thick] (n23.center) -- (n32.center);
        \draw[very thick] (n22.center) -- (n33.center);
        
        \draw[very thick] (n33.center) -- (n43.center);
        \draw[very thick] (n32.center) -- (n41.center);
        \draw[very thick] (n31.center) -- (n42.center);
		
        \draw[very thick] (n44.center) -- (n53.center);
        \draw[very thick] (n43.center) -- (n54.center);
        \draw[very thick] (n42.center) -- (n52.center);
        \draw[very thick] (n41.center) -- (n61.center);

        \draw[very thick] (n54.center) -- (n64.center);
        \draw[very thick] (n52.center) -- (n63.center);
        \draw[very thick] (n53.center) -- (n62.center);	
		
		\node[bigroundnodew] at (1.5,3.5) (m1) {};
		\node[bigroundnode] at (2.5,2.5) (m2) {};
		\node[bigroundnode] at (3.5,1.5) (m3) {};
		\node[bigroundnode] at (4.5,3.5) (m4) {};
		\node[bigroundnode] at (5.5,2.5) (m5) {};

	\end{scope}

	
\end{tikzpicture}}

\caption{A CW complex $\BLC_{z_1}$ which is homotopically equivalent to $\BL_{z_1}$.}
\label{fig:4312xx}
\end{figure}

The construction for $\BLC_{-z_1}$ is completely similar. 
Therefore $\BL_\sigma$ has $16$ connected components, all contractible.
\end{example}

\bigskip

Ancestries and strata of dimension $2$ also admit
a not too complicated description.
Understanding them allows us to compute
the fundamental group of each connected component.
For small values of $n$, when there are no ancestries (or strata)
of dimension $3$ or higher, we obtain a full description
of the spaces, in many cases allowing us to deduce
that their connected components are contractible.

Let $\varepsilon_0$ be a preancestry of dimension $2$ and $k_1 < k_2 < k_3 < k_4$, $|\varepsilon_0(k_i)|=2$. 
We prefer to break into cases, that we call type I and type II.
In any case, we have $\varepsilon_0(k_1)=-2$ and $\varepsilon_0(k_4)=+2$. 
An ancestry $\varepsilon$ of dimension $2$ is said to be of type I
when one of the following situations happens: 
if $\varepsilon_0(k_2)=+2$ then $\varepsilon_0(k_3)=-2$ and
$i_{k_1} = i_{k_2}$, \,
$i_{k_3} = i_{k_4}$; or
if $\varepsilon_0(k_2)=-2$ and $|i_{k_1} - i_{k_2}| > 1$, 
in this case the preancestry also has two pairs of consecutive intersections on two rows. 
In Figure~\ref{fig:types}, first and second diagrams are examples of preancestries of dimension $2$ of type I. 
An ancestry $\varepsilon_0$ of dimension $2$ is said to be of type II if $\varepsilon_0(k_2)=-2$ and
$|i_{k_1} - i_{k_2}| = 1$. In this case we have $i_{k_1} = i_{k_4}$, 
$i_{k_2} = i_{k_3}$; see third diagram of Figure~\ref{fig:types}.
A preancestry of dimension $2$ 
is either of type I or of type II.

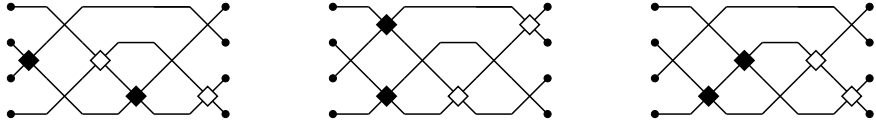
\begin{figure}[ht!]
\centering
\resizebox{0.8\textwidth}{!}{
\begin{tikzpicture}[roundnode/.style={circle, draw=black, fill=black, minimum size=0.1mm,inner sep=2pt},
squarednode/.style={rectangle, draw=red!60, fill=red!5, very thick, minimum size=5mm},
bigdiamondnode/.style={draw,diamond, fill=black, minimum size=1mm,very thick,inner sep=4pt},
diamondnodew/.style={draw,diamond, fill=white, minimum size=1mm,very thick,inner sep=4pt},
bigdiamondnodew/.style={draw,diamond, fill=white, minimum size=1mm,very thick, inner sep=4pt},
bigroundnode/.style={circle, draw=black, fill=black, minimum size=0.1mm,very thick, inner sep=4pt},
bigroundnodew/.style={circle, draw=black, fill=white, minimum size=0.1mm,very thick, inner sep=4pt},]
\begin{scope}[shift={(-9.5,0)}]
		\node[roundnode] at (1,4) (n14) {};
		\node[roundnode] at (1,3) (n13) {};
		\node[roundnode] at (1,2) (n12) {};
		\node[roundnode] at (1,1) (n11) {};
		
		\node at (2,4) (n24) {};
		\node at (2,3) (n23) {};
		\node at (2,2) (n22) {};
		\node at (2,1) (n21) {};

		\node at (3,4) (n34) {};
		\node at (3,3) (n33) {};
		\node at (3,2) (n32) {};
		\node at (3,1) (n31) {};
		
		\node at (4,4) (n44) {};
		\node at (4,3) (n43) {};
		\node at (4,2) (n42) {};
		\node at (4,1) (n41) {};

		\node at (5,4) (n54) {};
		\node at (5,3) (n53) {};
		\node at (5,2) (n52) {};
		\node at (5,1) (n51) {};
		
		\node at (6,4) (n64) {};
		\node at (6,3) (n63) {};
		\node at (6,2) (n62) {};
		\node at (6,1) (n61) {};

		
		\node[roundnode] at (7,4) (n74) {};
		\node[roundnode] at (7,3) (n73) {};
		\node[roundnode] at (7,2) (n72) {};
		\node[roundnode] at (7,1) (n71) {};
		
		\draw[very thick] (n14) -- (n24.center);
		\draw[very thick] (n13) -- (n22.center);
		\draw[very thick] (n12) -- (n23.center);
        \draw[very thick] (n11) -- (n21.center);		
        
        \draw[very thick] (n24.center) -- (n33.center);
        \draw[very thick] (n23.center) -- (n34.center);
        \draw[very thick] (n22.center) -- (n31.center);
        \draw[very thick] (n21.center) -- (n32.center);
        
        \draw[very thick] (n34.center) -- (n64.center);
        \draw[very thick] (n33.center) -- (n42.center);
        \draw[very thick] (n32.center) -- (n43.center);
        \draw[very thick] (n31.center) -- (n41.center);
		
        \draw[very thick] (n43.center) -- (n53.center);
        \draw[very thick] (n42.center) -- (n51.center);
       \draw[very thick] (n41.center) -- (n52.center);		
			
        \draw[very thick] (n54.center) -- (n64.center);
        \draw[very thick] (n53.center) -- (n62.center);
        \draw[very thick] (n52.center) -- (n63.center);
       \draw[very thick] (n51.center) -- (n61.center);	
        
				 \draw[very thick] (n64.center) -- (n73.center);
        \draw[very thick] (n63.center) -- (n74.center);
        \draw[very thick] (n62.center) -- (n71.center);
        \draw[very thick] (n61.center) -- (n72.center);

		\node[bigdiamondnode] at (1.5,2.5) (m1) {};
		\node[bigdiamondnodew] at (3.5,2.5) (m3) {};
		\node[bigdiamondnode] at (4.5,1.5) (m4) {};
		\node[bigdiamondnodew] at (6.5,1.5) (m6) {};
	
	\end{scope}

		\begin{scope}[shift={(-0.5,0)}]
		\node[roundnode] at (1,4) (n14) {};
		\node[roundnode] at (1,3) (n13) {};
		\node[roundnode] at (1,2) (n12) {};
		\node[roundnode] at (1,1) (n11) {};
		
		\node at (2,4) (n24) {};
		\node at (2,3) (n23) {};
		\node at (2,2) (n22) {};
		\node at (2,1) (n21) {};

		\node at (3,4) (n34) {};
		\node at (3,3) (n33) {};
		\node at (3,2) (n32) {};
		\node at (3,1) (n31) {};
		
		\node at (4,4) (n44) {};
		\node at (4,3) (n43) {};
		\node at (4,2) (n42) {};
		\node at (4,1) (n41) {};

		\node at (5,4) (n54) {};
		\node at (5,3) (n53) {};
		\node at (5,2) (n52) {};
		\node at (5,1) (n51) {};
		
		\node at (6,4) (n64) {};
		\node at (6,3) (n63) {};
		\node at (6,2) (n62) {};
		\node at (6,1) (n61) {};

		
		\node[roundnode] at (7,4) (n74) {};
		\node[roundnode] at (7,3) (n73) {};
		\node[roundnode] at (7,2) (n72) {};
		\node[roundnode] at (7,1) (n71) {};
		
		\draw[very thick] (n14) -- (n24.center);
		\draw[very thick] (n13) -- (n22.center);
		\draw[very thick] (n12) -- (n23.center);
        \draw[very thick] (n11) -- (n21.center);		
        
        \draw[very thick] (n24.center) -- (n33.center);
        \draw[very thick] (n23.center) -- (n34.center);
        \draw[very thick] (n22.center) -- (n31.center);
        \draw[very thick] (n21.center) -- (n32.center);
        
        \draw[very thick] (n34.center) -- (n64.center);
        \draw[very thick] (n33.center) -- (n42.center);
        \draw[very thick] (n32.center) -- (n43.center);
        \draw[very thick] (n31.center) -- (n41.center);
		
        \draw[very thick] (n43.center) -- (n53.center);
        \draw[very thick] (n42.center) -- (n51.center);
       \draw[very thick] (n41.center) -- (n52.center);		
			
        \draw[very thick] (n54.center) -- (n64.center);
        \draw[very thick] (n53.center) -- (n62.center);
        \draw[very thick] (n52.center) -- (n63.center);
       \draw[very thick] (n51.center) -- (n61.center);	
        
				 \draw[very thick] (n64.center) -- (n73.center);
        \draw[very thick] (n63.center) -- (n74.center);
        \draw[very thick] (n62.center) -- (n71.center);
        \draw[very thick] (n61.center) -- (n72.center);

		\node[bigdiamondnode] at (2.5,1.5) (m2) {};
		\node[bigdiamondnode] at (2.5,3.5) (m3) {};
		\node[bigdiamondnodew] at (4.5,1.5) (m4) {};
		\node[bigdiamondnodew] at (6.5,3.5) (m6) {};
	\end{scope}
	
	\begin{scope}[shift={(8.5,0)}]
		\node[roundnode] at (1,4) (n14) {};
		\node[roundnode] at (1,3) (n13) {};
		\node[roundnode] at (1,2) (n12) {};
		\node[roundnode] at (1,1) (n11) {};
		
		\node at (2,4) (n24) {};
		\node at (2,3) (n23) {};
		\node at (2,2) (n22) {};
		\node at (2,1) (n21) {};

		\node at (3,4) (n34) {};
		\node at (3,3) (n33) {};
		\node at (3,2) (n32) {};
		\node at (3,1) (n31) {};
		
		\node at (4,4) (n44) {};
		\node at (4,3) (n43) {};
		\node at (4,2) (n42) {};
		\node at (4,1) (n41) {};

		\node at (5,4) (n54) {};
		\node at (5,3) (n53) {};
		\node at (5,2) (n52) {};
		\node at (5,1) (n51) {};
		
		\node at (6,4) (n64) {};
		\node at (6,3) (n63) {};
		\node at (6,2) (n62) {};
		\node at (6,1) (n61) {};

		
		\node[roundnode] at (7,4) (n74) {};
		\node[roundnode] at (7,3) (n73) {};
		\node[roundnode] at (7,2) (n72) {};
		\node[roundnode] at (7,1) (n71) {};
		
		\draw[very thick] (n14) -- (n24.center);
		\draw[very thick] (n13) -- (n22.center);
		\draw[very thick] (n12) -- (n23.center);
        \draw[very thick] (n11) -- (n21.center);		
        
        \draw[very thick] (n24.center) -- (n33.center);
        \draw[very thick] (n23.center) -- (n34.center);
        \draw[very thick] (n22.center) -- (n31.center);
        \draw[very thick] (n21.center) -- (n32.center);
        
        \draw[very thick] (n34.center) -- (n64.center);
        \draw[very thick] (n33.center) -- (n42.center);
        \draw[very thick] (n32.center) -- (n43.center);
        \draw[very thick] (n31.center) -- (n41.center);
		
        \draw[very thick] (n43.center) -- (n53.center);
        \draw[very thick] (n42.center) -- (n51.center);
       \draw[very thick] (n41.center) -- (n52.center);		
			
        \draw[very thick] (n54.center) -- (n64.center);
        \draw[very thick] (n53.center) -- (n62.center);
        \draw[very thick] (n52.center) -- (n63.center);
       \draw[very thick] (n51.center) -- (n61.center);	
        
				 \draw[very thick] (n64.center) -- (n73.center);
        \draw[very thick] (n63.center) -- (n74.center);
        \draw[very thick] (n62.center) -- (n71.center);
        \draw[very thick] (n61.center) -- (n72.center);

		\node[bigdiamondnode] at (2.5,1.5) (m1) {};
		\node[bigdiamondnode] at (3.5,2.5) (m3) {};
		\node[bigdiamondnodew] at (5.5,2.5) (m4) {};
		\node[bigdiamondnodew] at (6.5,1.5) (m6) {};

	\end{scope}
\end{tikzpicture}}
\caption{Here we have three preancestries of dimension $2$: on the left and on the central are of type I and on the right of type II.}
\label{fig:types}
\end{figure}

\begin{example}
\label{example:45132-b}

Set $\sigma = a_2 a_3 a_1 a_2 a_4 a_3 a_2 = [4 5 1 3 2]$. 
In the notation of cycles, $\sigma = (143)(25)$;
therefore $n = 4$, $\ell = 7$, $c = 2$ and $b = 0$.
It follows from Equation~\ref{theo:N} that for $z \in \acute\sigma \Quat_5$,
we have $N(z) = 4 + 4\sqrt{2} \Re(z)$.
In addition,
\[ \acute\sigma = 
\frac{-1 + \hat a_1\hat a_2
+ \hat a_3 - \hat a_1\hat a_2\hat a_3
+ \hat a_1\hat a_4 + \hat a_2\hat a_4
-\hat a_1 \hat a_3 \hat a_4 - \hat a_2\hat a_3\hat a_4}{2\sqrt{2}}, \]
therefore $z \in \acute\sigma \Quat_5$ implies
$\Re(z) \in \{0, \pm \sqrt{2}/4\}$.

Next we chose the representatives of interest in each orbit.
If $z \in \Pi^{-1}[\{ \sigma \}]$ and $\Re(z) \neq 0 $ then $c_{\anti}(z)=0$;
if $\Re(z) = 0$ then $c_{\anti}(z) = 1$.
Therefore the set $\acute\sigma \Quat_5$ thus has $3$ orbits under $\cE_4$,
determined by the real part,
of sizes $8$, $16$ and $8$:
\begin{gather*}
\cO_{\acute\sigma},  \quad \Re(z) = -\frac{\sqrt{2}}{4}, \quad
N(z) =2, \quad N_{\thin}(z) = 2,  \\
\cO_{\hat a_1\acute\sigma}, \quad \Re(z) = 0, \quad
N(z) = 4, \quad N_{\thin}(z) = 0, \\
\cO_{-\acute\sigma}, \quad \Re(z) = \frac{\sqrt{2}}{4}, \quad
N(z) = 6, \quad N_{\thin}(z) = 0.
\end{gather*}

If $\Re(z)=-\frac{\sqrt{2}}{4}$, the set $\BL_z$ has two thin connected components 
(and no thick one). 
Moreover, the stratification of $\BL_{\hat a_1 \acute \sigma}$ is discussed in the Example~\ref{example:45132-a1};
we recall it is connected and contractible.
Now we detail the stratification of $\BL_{-\acute{\sigma}}$.
We start with the graphical representation in Figure~\ref{fig:45132xx}.

\begin{figure}[ht]
\centering

\resizebox{0.4\textwidth}{!}{
}

\caption{The stratification of $\BL_{-\acute{\sigma}}$.}
\label{fig:45132xx}
\end{figure}

In $\BL_{- \acute{\sigma}}$ there are $6$ strata of dimension $0$,
$6$ strata of dimension $1$ and
exactly one stratum of dimension $2$:
with ancestry $\varepsilon = (-2,-2,-1,-1,+1,+2,+2)$
and corresponding $\xi$-ancestry $\xi = (1,1,0,2,0,1,1)$.
The ancestry $\varepsilon$ is of type II.

To get the stratification above we perform the computations in the orthogonal group.
In Section~\ref{section:preliminaries} we defined the one parameter subgroups
$\alpha_i:  \RR \to \SO_{n+1}$, given by

\[
\alpha_i(\theta) =
\begin{pmatrix} I & & & \\
& \cos(\theta) & -\sin(\theta) & \\
& \sin(\theta) & \cos(\theta) & \\
& & & I \end{pmatrix},\]
where the central block ocupies rows and columns
$i$ and $i+1$. 
To make computations more algebraic, let $t = \tan\left(\frac{\theta}{2}\right)$.
Therefore 

\[
\zeta_i(t)=\alpha_i(2\arctan(t))=
\begin{pmatrix} I & & & \\
& \frac{1-t^2}{1+t^2} & -\frac{2t}{1+t^2} & \\\\
& \frac{2t}{1+t^2} & \frac{1-t^2}{1+t^2} & \\
& & & I \end{pmatrix}.\]

In order to study a transversal section to $\BL_{\varepsilon}$,
we take
\[ z_{7} = \zeta_2(-1+x_1) \zeta_3(-1+x_2) \zeta_1(-1/2) \zeta_2(-1/2) \zeta_4(1/2) \zeta_3(1/2) \zeta_2(1/2).\]
In order to determine the position
of a point in the strata above,
we must study the signs of 
\begin{gather*}
p_1(x_1, x_2)=x_1, \quad p_2(x_2, x_2)=x_2 \\
p_3(x_1,x_2)=5 x^2_1 x_2^2 - 10 x_1^2 x_2 - 2 x_1 x_2^2 + 10 x_1^2 + 4 x_1 x_2 - 8 x_2^2 - 20 x_1 + 16 x_2.
\end{gather*}

These three expressions have pairwise linearly independent deritatives
in the origin.
The transversal section is shown in Figure~\ref{fig:codim2II3}.
Thus, in the CW complex shown in Figure~\ref{fig:45132xx},
the cell of dimension $2$ glues in the obvious way.
The set $\BL_{-\hat a_1\acute\eta}$ is therefore contractible.

\begin{figure}[ht]
\begin{center}
\includegraphics[scale=0.5]{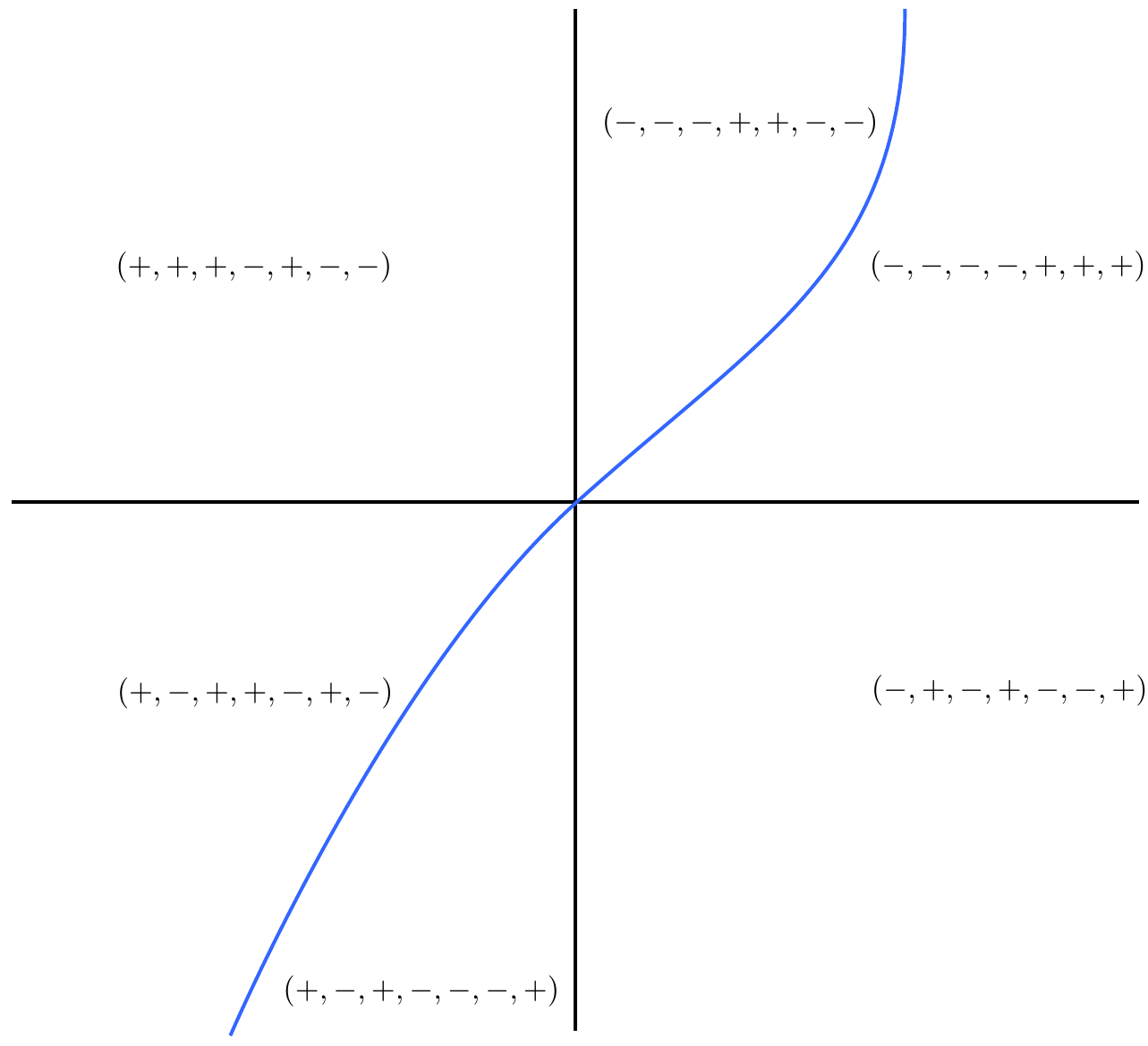}
\end{center}
\caption{A transversal section to a stratum of dimension $2$.}
\label{fig:codim2II3}
\end{figure}


Summing up, $\BL_{\sigma}$ has $40$ connected components, 
all contractible.
\end{example}




\label{section:n4}

\section{Higher dimension strata}
\label{section:codimtwo}

In this section we will present examples that contain 
higher dimensional strata.
We will not be presenting theorems in the most general form but we hope through examples 
to convey the ideas of the theory.
The last example we present is the most important for $n=4$ and the computations help to clarify the theory;
in this example there appear strata of dimension up to $4$. \\

\begin{example}


Consider $\sigma = [4 3 5 2 1]$. Let us fixed the reduced word: $\sigma=a_1 a_3 a_2 a_1 a_4 a_3 a_2 a_1$ 
(see Figure~\ref{fig:43521}).



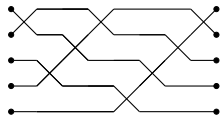
\begin{figure}[h!]
\centering
\resizebox{0.2\textwidth}{!}{
\begin{tikzpicture}[roundnode/.style={circle, draw=black, fill=black, minimum size=0.1mm,inner sep=2pt},
squarednode/.style={rectangle, draw=red!60, fill=red!5, very thick, minimum size=5mm},
diamondnode/.style={draw,diamond, fill=black, minimum size=1mm,very thick,inner sep=4pt},
diamondnodew/.style={draw,diamond, fill=white, minimum size=1mm,very thick,inner sep=4pt},
bigroundnode/.style={circle, draw=black, fill=black, minimum size=0.1mm,very thick, inner sep=4pt},
bigroundnodew/.style={circle, draw=black, fill=white, minimum size=0.1mm,very thick, inner sep=4pt},]

	
	\begin{scope}[shift={(-5.5,0)}]
		\node[roundnode] at (1,5) (n15) {};
		\node[roundnode] at (1,4) (n14) {};
		\node[roundnode] at (1,3) (n13) {};
		\node[roundnode] at (1,2) (n12) {};
		\node[roundnode] at (1,1) (n11) {};
		
        \node at (2,5) (n25) {};		
		\node at (2,4) (n24) {};
		\node at (2,3) (n23) {};
		\node at (2,2) (n22) {};
		\node at (2,1) (n21) {};

		\node at (3,5) (n35) {};
		\node at (3,4) (n34) {};
		\node at (3,3) (n33) {};
		\node at (3,2) (n32) {};
		\node at (3,1) (n31) {};
		
        \node at (4,5) (n45) {};		
		\node at (4,4) (n44) {};
		\node at (4,3) (n43) {};
		\node at (4,2) (n42) {};
		\node at (4,1) (n41) {};

		\node at (5,5) (n55) {};		
		\node at (5,4) (n54) {};
		\node at (5,3) (n53) {};
		\node at (5,2) (n52) {};
		\node at (5,1) (n51) {};
		
		\node at (6,5) (n65) {};		
		\node at (6,4) (n64) {};
		\node at (6,3) (n63) {};
		\node at (6,2) (n62) {};
		\node at (6,1) (n61) {};

		\node at (7,5) (n75) {};		
		\node at (7,4) (n74) {};
		\node at (7,3) (n73) {};
		\node at (7,2) (n72) {};
		\node at (7,1) (n71) {};

		\node at (8,5) (n85) {};		
		\node at (8,4) (n84) {};
		\node at (8,3) (n83) {};
		\node at (8,2) (n82) {};
		\node at (8,1) (n81) {};
		\node[roundnode] at (9,5) (n95) {};
		\node[roundnode] at (9,4) (n94) {};
		\node[roundnode] at (9,3) (n93) {};
		\node[roundnode] at (9,2) (n92) {};
		\node[roundnode] at (9,1) (n91) {};
		
		\draw[very thick] (n15) -- (n24.center);
		\draw[very thick] (n14) -- (n25.center);
		\draw[very thick] (n13) -- (n23.center);
		\draw[very thick] (n12) -- (n22.center);
        \draw[very thick] (n11) -- (n21.center);		
        
        \draw[very thick] (n25.center) -- (n45.center);
        \draw[very thick] (n24.center) -- (n34.center);
        \draw[very thick] (n23.center) -- (n32.center);
        \draw[very thick] (n22.center) -- (n33.center);
        \draw[very thick] (n21.center) -- (n51.center);
        
        \draw[very thick] (n34.center) -- (n43.center);
        \draw[very thick] (n33.center) -- (n44.center);
        \draw[very thick] (n32.center) -- (n52.center);
		
        \draw[very thick] (n25.center) -- (n45.center);
        \draw[very thick] (n24.center) -- (n34.center);
        \draw[very thick] (n23.center) -- (n32.center);
        \draw[very thick] (n22.center) -- (n33.center);
        \draw[very thick] (n21.center) -- (n51.center);

        \draw[very thick] (n32.center) -- (n52.center);
        
        \draw[very thick] (n33.center) -- (n44.center);
        
        \draw[very thick] (n34.center) -- (n43.center);	
        
        \draw[very thick] (n35.center) -- (n45.center);	
        
        \draw[very thick] (n42.center) -- (n52.center);
        
        \draw[very thick] (n43.center) -- (n63.center);
        
        \draw[very thick] (n44.center) -- (n55.center);			
        
        \draw[very thick] (n45.center) -- (n54.center);
        
        \draw[very thick] (n51.center) -- (n62.center);
        
        \draw[very thick] (n52.center) -- (n61.center);	
        
        \draw[very thick] (n53.center) -- (n63.center);	        
        				
		\draw[very thick] (n54.center) -- (n74.center);	
		
		\draw[very thick] (n55.center) -- (n85.center);

	    \draw[very thick] (n61.center) -- (n91.center);
	    
	    \draw[very thick] (n62.center) -- (n73.center);	
	    
	    \draw[very thick] (n63.center) -- (n72.center);
	    
	    \draw[very thick] (n62.center) -- (n73.center);	
	    
	    \draw[very thick] (n64.center) -- (n74.center);
	    
	    \draw[very thick] (n72.center) -- (n92.center);
	    
	    \draw[very thick] (n73.center) -- (n84.center);	
	    
	    \draw[very thick] (n74.center) -- (n83.center);	
	    
	    \draw[very thick] (n83.center) -- (n93.center);	
	    
	    \draw[very thick] (n85.center) -- (n94.center);
	    
	    \draw[very thick] (n84.center) -- (n95.center);

	\end{scope}
\end{tikzpicture}	}
\caption{The permutation $\sigma \in S_5$.}
\label{fig:43521}
\end{figure} 

\medskip

We have $\ell = 8$, $b = |\Block(\sigma)| = 0$, $c = \nc(\sigma) = 1$ and
\begin{gather*}
\acute{\sigma}=
\frac{1}{4} \bigg(
-1 - \hat a_1 + \hat a_2  - \hat a_1 \hat a_2 - \hat a_3 + \hat a_1 \hat a_3
- \hat a_2 \hat a_3 - \hat a_1 \hat a_2 \hat a_3 - \hat a_4 \qquad \qquad \\ \qquad \qquad
+ \hat a_1 \hat a_4 + \hat a_2 \hat a_4 + \hat a_1 \hat a_2 \hat a_4
- \hat a_3 \hat a_4  - \hat a_1 \hat a_3 \hat a_4 - \hat a_2 \hat a_3 \hat a_4
+ \hat a_1 \hat a_2 \hat a_3 \hat a_4
\bigg).
\end{gather*}


%

It follows from Equation~\ref{theo:N} that
$N(z) = 8 + 8 \Re(z)$.
In this example, it turns out that $\Pi^{-1}[\{\sigma\}]$ contains
$16$ elements with $\Re(z) = \frac14$ and
$16$ elements with $\Re(z) = -\frac14$.

It turns out that, for all $z \in \Pi^{-1}[\{\sigma\}]$, $\Re(z) \neq 0$ and then $c_{\anti}(z) = 1$.
The set $\Pi^{-1}[\{\sigma\}]$ thus has $2$ orbits both of $16$ sizes.

\begin{align*}
\cO_{\acute\sigma},
\quad \Re(z)=-\frac14, \quad N(z) = 6, \quad  
\quad N_{\thin}(z) = 1,  \\
\cO_{\hat a_1 \hat a_2 \hat a_3 \acute{\sigma}},\quad \Re(z)=\frac14, \quad N(z) = 10, 
\quad N_{\thin}(z) = 0,  \\
\end{align*}

The set $\BL_{\hat a_1 \hat a_2 \hat a_3 \acute\sigma}$ is connected.
But the set $\BL_{\acute\sigma}$ has two connected components, the thick part of $\BL_{ \acute\sigma}$ is also connected.
In Figure~\ref{fig:54231BLacutesigma} and Figure~\ref{fig:54231BLacuteminussigma} we draw the CW complex for one representative for each orbit. 
The total number of connected components of $\BL_{\sigma}$ is therefore $48$.
Moreover all these connected components are contractible. 
\end{example}


\begin{figure}[h!]

\centering
    \resizebox{0.6\textwidth}{!}{

}
\caption{The CW complex $\BLC_{\hat a_1 \hat a_2 \hat a_3 \acute{\sigma}}$.}
\label{fig:54231BLacutesigma}
\end{figure}

\begin{figure}[h!]
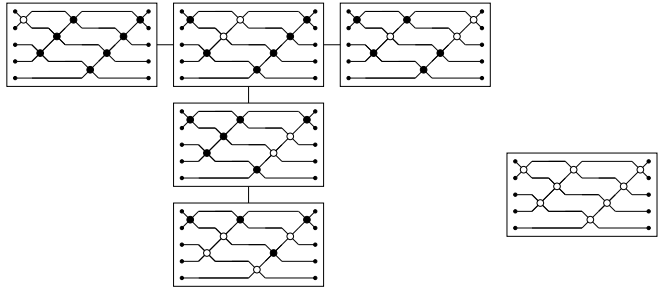


\centering
    \resizebox{0.6\textwidth}{!}{

%
}
\caption{The CW complex $\BLC_{\acute{\sigma}}$.}
\label{fig:54231BLacuteminussigma}
\end{figure}

\newpage


\begin{example}
\label{example:54321}

Set  $\sigma = \eta = a_1a_2a_1a_3a_2a_1a_4a_3a_2a_1$; Figure \ref{diagrampermutation54321} shows this reduced word as a diagram.

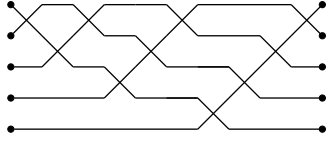
\begin{figure}[ht!]
\centering
 \resizebox{0.30\textwidth}{!}{

\begin{tikzpicture}[roundnode/.style={circle, draw=black, fill=black, minimum size=0.1mm,inner sep=2pt},
squarednode/.style={rectangle, draw=red!60, fill=red!5, very thick, minimum size=5mm},
diamondnode/.style={draw,diamond, fill=black, minimum size=1mm,very thick,inner sep=4pt},
diamondnodew/.style={draw,diamond, fill=white, minimum size=1mm,very thick,inner sep=4pt},
bigroundnode/.style={circle, draw=black, fill=black, minimum size=0.1mm,very thick, inner sep=4pt},
bigroundnodew/.style={circle, draw=black, fill=white, minimum size=0.1mm,very thick, inner sep=4pt},]

	
	\begin{scope}[shift={(-5.5,0)}]
		\node[roundnode] at (1,5) (n15) {};
		\node[roundnode] at (1,4) (n14) {};
		\node[roundnode] at (1,3) (n13) {};
		\node[roundnode] at (1,2) (n12) {};
		\node[roundnode] at (1,1) (n11) {};
		
        \node at (2,5) (n25) {};		
		\node at (2,4) (n24) {};
		\node at (2,3) (n23) {};
		\node at (2,2) (n22) {};
		\node at (2,1) (n21) {};

		\node at (3,5) (n35) {};
		\node at (3,4) (n34) {};
		\node at (3,3) (n33) {};
		\node at (3,2) (n32) {};
		\node at (3,1) (n31) {};
		
        \node at (4,5) (n45) {};		
		\node at (4,4) (n44) {};
		\node at (4,3) (n43) {};
		\node at (4,2) (n42) {};
		\node at (4,1) (n41) {};

		\node at (5,5) (n55) {};		
		\node at (5,4) (n54) {};
		\node at (5,3) (n53) {};
		\node at (5,2) (n52) {};
		\node at (5,1) (n51) {};
		
		\node at (6,5) (n65) {};		
		\node at (6,4) (n64) {};
		\node at (6,3) (n63) {};
		\node at (6,2) (n62) {};
		\node at (6,1) (n61) {};

		\node at (7,5) (n75) {};		
		\node at (7,4) (n74) {};
		\node at (7,3) (n73) {};
		\node at (7,2) (n72) {};
		\node at (7,1) (n71) {};

		\node at (8,5) (n85) {};		
		\node at (8,4) (n84) {};
		\node at (8,3) (n83) {};
		\node at (8,2) (n82) {};
		\node at (8,1) (n81) {};
		
        \node at (9,5) (n95) {};		
		\node at (9,4) (n94) {};
		\node at (9,3) (n93) {};
		\node at (9,2) (n92) {};
		\node at (9,1) (n91) {};
		
		
		\node at (10,5) (n105) {};		
		\node at (10,4) (n104) {};
		\node at (10,3) (n103) {};
		\node at (10,2) (n102) {};
		\node at (10,1) (n101) {};	
		
		
		\node[roundnode] at (11,5) (n115) {};
		\node[roundnode] at (11,4) (n114) {};
		\node[roundnode] at (11,3) (n113) {};
		\node[roundnode] at (11,2) (n112) {};
		\node[roundnode] at (11,1) (n111) {};
		
		\draw[very thick] (n15) -- (n24.center);
		\draw[very thick] (n14) -- (n25.center);
		\draw[very thick] (n13) -- (n23.center);
		\draw[very thick] (n12) -- (n42.center);
        \draw[very thick] (n11) -- (n71.center);		
        
        \draw[very thick] (n25.center) -- (n35.center);
        \draw[very thick] (n24.center) -- (n33.center);
        \draw[very thick] (n23.center) -- (n34.center);
        
        \draw[very thick] (n35.center) -- (n44.center);
        \draw[very thick] (n34.center) -- (n45.center);
        \draw[very thick] (n33.center) -- (n43.center);
		
        \draw[very thick] (n45.center) -- (n55.center);
        \draw[very thick] (n44.center) -- (n54.center);
        \draw[very thick] (n43.center) -- (n52.center);
        \draw[very thick] (n42.center) -- (n53.center);

        \draw[very thick] (n55.center) -- (n65.center);
        
        \draw[very thick] (n54.center) -- (n63.center);
        
        \draw[very thick] (n53.center) -- (n64.center);	
        
        \draw[very thick] (n52.center) -- (n72.center);

        \draw[very thick] (n65.center) -- (n74.center);
        
        \draw[very thick] (n64.center) -- (n75.center);
        
        \draw[very thick] (n63.center) -- (n83.center);			
        
        \draw[very thick] (n62.center) -- (n72.center);

        \draw[very thick] (n75.center) -- (n105.center);
        
        \draw[very thick] (n74.center) -- (n94.center);	
        
        \draw[very thick] (n73.center) -- (n83.center);	        
        				
		\draw[very thick] (n72.center) -- (n81.center);	
		
		\draw[very thick] (n71.center) -- (n82.center);

	    \draw[very thick] (n83.center) -- (n92.center);
	    
	    \draw[very thick] (n82.center) -- (n93.center);

	    \draw[very thick] (n94.center) -- (n103.center);
	    
	    \draw[very thick] (n93.center) -- (n104.center);	
	    
	    \draw[very thick] (n92.center) -- (n112.center);
	    
	    \draw[very thick] (n81.center) -- (n111.center);

	    \draw[very thick] (n103.center) -- (n113.center);	
	    
	    \draw[very thick] (n104.center) -- (n115.center);	
	    
	    \draw[very thick] (n105.center) -- (n114.center);

	\end{scope}
\end{tikzpicture}	}
\label{diagrampermutation54321}
\caption{The permutation $\eta \in S_5$.}
\end{figure}

In the notations of cycles $\eta = (15)(24)(3)$. Therefore, we have $\ell = 10$, $b = |\Block(\eta)| = 0$, $c = \nc(\eta) = 3$. Moreover,
$$ \Pi^{-1} [\{ \eta \}] = \bigg\{ \frac{\pm 1 \pm \hat a_2\hat a_3
\pm \hat a_1\hat a_4 \pm \hat a_1 \hat a_2\hat a_3\hat a_4}{2},
\frac{\pm \hat a_1 \pm \hat a_1\hat a_2\hat a_3 \pm \hat a_4 \pm \hat a_2\hat a_3\hat a_4}{2},$$
$$\frac{\pm \hat a_1 \hat a_2 \pm \hat a_1 \hat a_3 \pm \hat a_2 \hat a_4
\pm \hat a_1 \hat a_4}{2}, \frac{ \pm \hat a_2 \pm \hat a_3 \pm \hat a_1  \hat a_2 \hat a_4 \pm \hat a_1  \hat a_3 \hat a_4}{2} \bigg \}$$
where we must take an even number of `$-$' signs.
Therefore, $\Pi^{-1}[\{\eta\}]$ contains
$4$ elements with $\Re(z) = \frac12$,
$4$ elements with $\Re(z) = -\frac12$ and
$24$ elements with $\Re(z) = 0$ (so that $|\Pi^{-1}[\{\eta\}]| = 32$).
From Equation~\ref{theo:N}, we have
$N(z) = 32 + 16 \Re(z)$.
It turns out that, for $z \in \Pi^{-1}[\{\eta\}]$,
$c_{\anti}(z) = 1$ if and only if $\Re(z) = 0$.
Therefore from the Remark~\ref{remark_orbits}, the set $\Pi^{-1}[\{\eta\}]$ has orbits of size $8$ and $4$.
If $\Re(z)=0$ then the orbit $\cO_z$ has cardinality $2^{n-c+2}=2^{4-3+2}=8$,
and if $\Re(z)\neq0$ then the orbit $\cO_z$ has cardinality $2^{n-c+1}=2^{4-3+1}=4$.
The action of $\cE_4$ splits the set $\acute\eta\Quat_{n+1}$
into $5$ orbits, of sizes $8, 4, 4, 8, 8$, shown below.
\begin{align*}
\cO_{\acute\eta} &= \left\{
\frac{\pm \hat a_1 \pm \hat a_1\hat a_2\hat a_3
\pm \hat a_4 \pm \hat a_2\hat a_3\hat a_4}{2} 
\right\},
\quad N(z) = 32,
\quad N_{\thin}(z) = 2,  \\
\cO_{\hat a_1\acute\eta} &= \left\{
\frac{1 \pm \hat a_2\hat a_3
\pm \hat a_1\hat a_4 \pm \hat a_1\hat a_2\hat a_3\hat a_4}{2}
\right\}, \quad N(z) = 40, 
\quad N_{\thin}(z) = 0,  \\
\cO_{-\hat a_1\acute\eta} &= \left\{
\frac{-1 \pm \hat a_2\hat a_3
\pm \hat a_1\hat a_4 \pm \hat a_1\hat a_2\hat a_3\hat a_4}{2}
\right\}, \quad N(z) = 24,
\quad N_{\thin}(z) = 0,  \\
\cO_{\hat a_2 \acute\eta} &= \left\{
\frac{\pm \hat a_1 \hat a_2 \pm \hat a_1 \hat a_3
\pm \hat a_2\hat a_4 \pm \hat a_1\hat a_4}{2}
\right\}, \quad N(z) = 32, 
\quad N_{\thin}(z) = 0,  \\
\cO_{\hat a_1\hat a_2\acute\eta} &= \left\{
\frac{\pm \hat a_2 \pm \hat a_3
\pm \hat a_1\hat a_2\hat a_4 \pm \hat a_1\hat a_3\hat a_4}{2}
\right\}, \quad N(z) = 32,
\quad N_{\thin}(z) = 0.  
\end{align*}

%
%

In order to count connected components
and obtain further information about the topology
of the sets $\BL_z$, $z \in \acute\eta\Quat_{n+1}$,
we can pick one representative from each orbit
and draw the strata.
In~\cite{Alves-Saldanha}, we already constructed the CW complex for $z = - \hat{a}_1 \acute{\eta} = - \acute{\eta}\hat{a}_4$ 
(see~\cite{Alves-Saldanha} to more details about this construction).
 $\BL_{-\hat{a}_1 \acute{\eta}}$ is homotopically equivalent to the disjoint union of two points.
In other words, the two connected components 
of $\BL_{-\hat a_1\acute\eta}$ are contractible.
Therefore the orbit $\cO_{-\hat{a}_1 \acute{\eta}}$ contributes with $8$ connected components to $\BL_{\eta}$.


Now we want to explore the decomposition into strata of $\BL_{\acute{\eta}}$.
Here instead of $+ 1$, $-1$, $+2$ and $- 2$ we use $\smawc$, $\smabc$, $\medwd$ and $\medbd$, respectively. 
Notice the similarity with diagrams.
For example, we use $(\medbd \smawc \medwd \medbd \smawc \smawc \smabc \medwd \smawc \smabc)$
instead of $(-2, +1, +2, -2, +1, +1, -1, +2, +1, -1)$.
A computations shows that $\BL_{\acute{\eta}}$ has: $32$ strata of dimension $0$ (two of them are thin and obviously contractible), $48$ strata of dimension $1$, $22$ strata of dimension $2$, $3$ strata of dimension $3$ and no strata of higher dimension. There are exactly three ancestries of dimension $3$:

\begin{gather*}
(\medbd \medbd \medbd \smabc \medwd \medwd \smabc \smawc \medwd \smabc), \\
(\smabc \medbd \medbd \medbd \smawc \medwd \smabc \medwd \medwd \smabc), \\
(\smabc \medbd \medbd \smabc \medbd \medwd \smabc \smabc \medwd \medwd).
\end{gather*}

In what follows we constructed the CW complex of the ancestry \\
$(\medbd \medbd \medbd \smabc \medwd \medwd \smabc \smawc \medwd \smabc)$. 
A sketch of these $3$-dimension cell appears in the Figure \ref{54321_dimension3_dot}; and in the
Figure \ref{54321_flattening} we represent a flattening of these cell.
Note that the outer hexagon in the Figure \ref{54321_flattening} correspond to the $2$-dimension cell 
with ancestry $(\smawc \medbd \medbd \smawc \smabc \medwd \smabc \smawc \medwd \smabc)$.

The construction of the other two cells of dimension $3$ are similar.
These three cells of dimension $3$ have in common the $2$-dimension cell with ancestry 
$(\smabc \medbd \medbd \smabc \smawc \medwd \smabc \smawc \medwd \smabc)$. 
In this hexagon we glue, in the obvious way, the three cells of dimension $3$. 
Therefore this connected component is contractible.

\begin{figure}[ht]
\begin{center}
\includegraphics[scale=0.4]{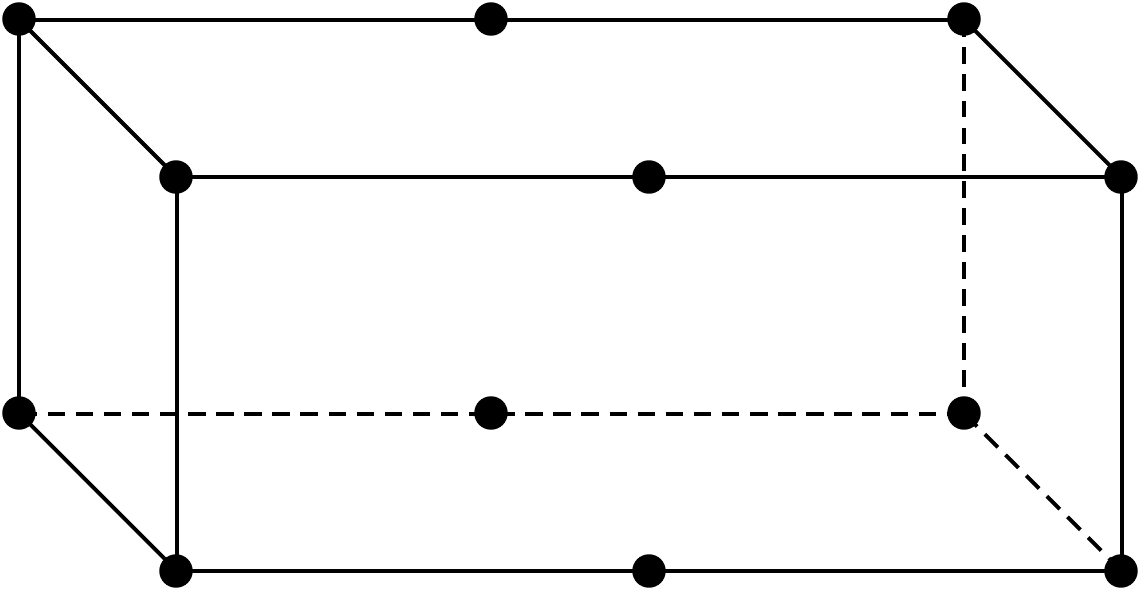}
\caption{A sketch of one $3$-dimensional cell in $\BL_{\acute{\eta}}$.}
\label{54321_dimension3_dot}
\end{center}

\end{figure}

\begin{figure}[h!]
\centering
 \resizebox{0.6\textwidth}{!}{

	}
\caption{ A flattening of the cell above.}
\label{54321_flattening}
\end{figure}


We did the construction of the CW complex of one representative of the orbits 
$\cO_{-\hat{a}_1 \acute{\eta}}$ (see~\cite{Alves-Saldanha}) and $\cO_{\acute{\eta}}$.
The construction of the CW complex for one representative of the orbits 
$\cO_{\hat{a}_2 \acute{\eta}}$ and $\cO_{\hat{a}_1 \hat{a}_2\acute{\eta}}$ 
is similar and simpler to the construction above and left to the reader.

Finally we will discuss the case of $\BL_{\hat{a}_1\acute{\eta}}$.
This orbit has 
$40$ strata of dimension $0$, 
$72$ strata of dimension $1$,
$42$ strata of dimension $2$,
$10$ strata of dimension $3$ and
$1$ stratum of dimension $4$. 
The CW complex associated with $\BL_{\hat a_1 \eta}$ is homotopically equivalent to a $4$-dimensional disk, $\DD^4$.
As a consequence it is connected and contractible. 
We briefly describe the construction of such CW complex in what follows.
First, a solid torus arises from the glueing of $6$ out of the $10$ cells of dimension $3$.
And then two of the $3$-cells fill in the boundary of the solid torus, 
which leads to a CW complex homotopically equivalent to $\Ss^2$.
The remaining $3$-dimensional cells fill in the boundary of $\Ss^2$, yielding $\Ss^3$.
Finally, the $4$-dimensional cell attaches to the previous construction leading to $\DD^4$.

Because a graphical representation of this construction is too complicated, we exhibit next 
a sequence of collapses starting with the initial CW complex for $\BL_{\hat a_1 \eta}$ 
and ending with a CW complex which is homeomorphic to a disk of dimension $2$, shown in Figure~\ref{fig:54321x}.
Each elementary collapse is described by a pair of cells of adjacent dimensions. 
The reader will then have difficulty obtaining a (rather long!) sequence of elementary collapses ending with a point.



\[ (\medbd \medbd \medbd \medbd \smabc \medwd \smawc \medwd \medwd \medwd,\medbd \medbd \smawc \medbd \smawc \smabc \smawc \medwd \medwd \medwd) \]
\[ (\smabc \medbd \medbd \medbd \smabc \medwd \smawc \medwd \medwd \smawc,\smabc \medbd \smawc \medbd \smawc \smabc \smawc \medwd \medwd \smawc) \]
\[ (\medbd \medbd \smabc \medbd \smabc \smawc \smawc \medwd \medwd \medwd,\smabc \medbd \smabc \medbd \smabc \smawc \smawc \medwd \medwd \smawc) \]
\[ (\medbd \smawc \medwd \medbd \smawc \medbd \smabc \medwd \smabc \medwd,\medbd \smawc \medwd \smabc \smawc \medbd \smabc \smawc \smabc \medwd) \]
\[ (\medbd \smawc \medwd \medbd \smawc \smawc \smabc \medwd \smawc \smabc,\medbd \smawc \medwd \smabc \smawc \smawc \smabc \smawc \smawc \smabc) \]
\[ (\smabc \smawc \smawc \medbd \smawc \medbd \smabc \medwd \smabc \medwd,\smabc \smawc \smawc \medbd \smawc \smawc \smabc \medwd \smawc \smabc) \]
\[ (\smabc \smawc \smawc \smabc \smawc \medbd \smabc \smawc \smabc \medwd,\smabc \smawc \smawc \smabc \smawc \smawc \smabc \smawc \smawc \smabc)\]
\[ (\medbd \smabc \medwd \medbd \smawc \medbd \smawc \medwd \smawc \medwd,\medbd \smabc \medwd \smabc \smawc \medbd \smawc \smawc \smawc \medwd) \]
\[ (\smawc \smawc \smabc \medbd \smawc \medbd \smawc \medwd \smawc \medwd,\smawc \smawc \smabc \smabc \smawc \medbd \smawc \smawc \smawc \medwd) \]
\[ (\medbd \smabc \medwd \medbd \smawc \smabc \smawc \medwd \smawc \smawc,\smawc \smawc \smabc \medbd \smawc \smabc \smawc \medwd \smawc \smawc) \]
\[ (\medbd \smabc \medwd \smabc \smawc \smabc \smawc \smawc \smawc \smawc,\smawc \smawc \smabc \smabc \smawc \smabc \smawc \smawc \smawc \smawc) \]
We now proceed to remove all remaining cells of dimension $3$.
\[ (\medbd \medbd \medbd \smawc \medwd \medwd \smabc \smabc \medwd \smawc,\smabc \medbd \medbd \smawc \smawc \medwd \smabc \smabc \medwd \smawc) \]
\[ (\smawc \medbd \medbd \smabc \medbd \medwd \smabc \smabc \medwd \medwd,\smawc \medbd \medbd \smabc \smabc \medwd \smabc \smabc \medwd \smawc) \]
\[ (\smawc \medbd \medbd \medbd \smawc \medwd \smabc \medwd \medwd \smabc,\smawc \medbd \medbd \smabc \smawc \medwd \smabc \smawc \medwd \smabc) \]
\[ (\medbd \medbd \medbd \smabc \medwd \medwd \smawc \smabc \medwd \smabc,\smawc \medbd \medbd \smawc \smabc \medwd \smawc \smabc \medwd \smabc) \]
\[ (\smabc \medbd \medbd \smabc \medbd \medwd \smawc \smawc \medwd \medwd,\smabc \medbd \medbd \smabc \smawc \medwd \smawc \smabc \medwd \smabc) \]
At this stage we have a CW complex of dimension $2$.

\begin{figure}[p]
\begin{center}
\includegraphics[scale=0.25]{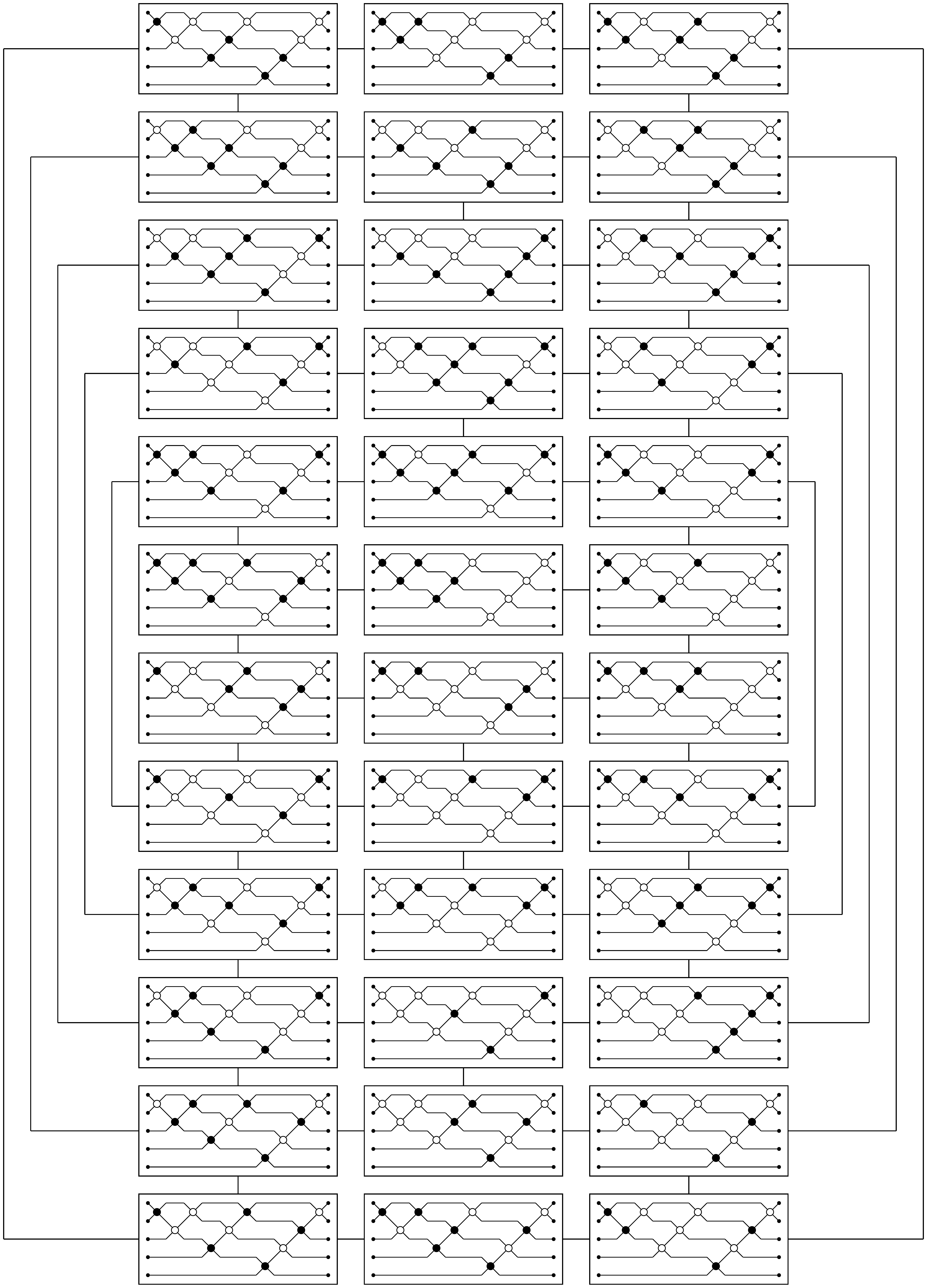}
\end{center}
\caption{The stratification of a CW complex homotopically equivalent to $\BL_{-\hat a_1\acute\eta}$, obtained after a sequence of collapses. 
This CW complex is homeomorphic to a disk of dimension $2$.}
\label{fig:54321x}
\end{figure}

\[ (\smabc \smabc \medbd \medbd \smabc \medwd \smawc \medwd \smawc \smawc,\smabc \smabc \smabc \medbd \smabc \smawc \smawc \medwd \smawc \smawc) \]
\[ (\smabc \smawc \medbd \medbd \smabc \medwd \smabc \medwd \smabc \smawc,\smabc \smawc \smabc \medbd \smabc \smawc \smabc \medwd \smabc \smawc) \] 
\[ (\medbd \smawc \medwd \medbd \smawc \smabc \smabc \medwd \smabc \smawc,\smabc \smawc \smawc \medbd \smawc \smabc \smabc \medwd \smabc \smawc) \]
\[ (\smawc \smabc \smabc \medbd \smawc \medbd \smabc \medwd \smabc \medwd,\smawc \smabc \smabc \medbd \smawc \smabc \smabc \medwd \smabc \smawc) \]
\[ (\medbd \smawc \medwd \smawc \smabc \medbd \smawc \smabc \smabc \medwd,\medbd \smawc \medwd \smawc \smabc \smabc \smawc \smabc \smabc \smawc) \]
\[ (\smawc \smabc \smabc \smawc \smabc \medbd \smawc \smabc \smabc \medwd,\smawc \smabc \smabc \smawc \smabc \smabc \smawc \smabc \smabc \smawc) \]
\[ (\smabc \smabc \smawc \medbd \smawc \medbd \smawc \medwd \smawc \medwd,\smabc \smabc \smawc \medbd \smawc \smabc \smawc \medwd \smawc \smawc) \] 
\[ (\medbd \smabc \medwd \medbd \smawc \smawc \smawc \medwd \smabc \smabc,\smabc \smabc \smawc \medbd \smawc \smawc \smawc \medwd \smabc \smabc) \]
\[ (\medbd \smabc \medwd \smawc \smabc \medbd \smabc \smabc \smawc \medwd,\smabc \smabc \smawc \smawc \smabc \medbd \smabc \smabc \smawc \medwd) \]
\[ (\medbd \smabc \medwd \smawc \smabc \smawc \smabc \smabc \smabc \smabc,\smabc \smabc \smawc \smawc \smabc \smawc \smabc \smabc \smabc \smabc)  \]

At this point, the CW complex is a disk of dimension $2$, as shown in Figure~\ref{fig:54321x}.
The reader is invited to complete the very long sequence of collapses.
This sequence of elementary collapses ends with a single vertex, as desired. 
This completes the proof that all $52$ connected components of $\BL_{\eta}$ are contractible.
\end{example}

\bigskip

\end{document}